\documentclass[a4paper,11pt]{article}

\usepackage{ucs}
\usepackage[latin9]{inputenc}
\usepackage{graphicx}
\usepackage{amsfonts}
\usepackage{dsfont}
\usepackage{amssymb}
\usepackage{amsmath}
\usepackage{amsthm}
\usepackage{mathrsfs}
\usepackage{enumerate}
\usepackage{stmaryrd}
\usepackage{fullpage}
\usepackage{ifthen}
\usepackage{subfigure}
\usepackage{epic}
\usepackage{authblk}
\usepackage{textcomp}
\usepackage[hypertexnames=false,colorlinks=true,linkcolor=blue,citecolor=blue]{hyperref}
\usepackage[numbers,comma,square,sort&compress]{natbib}

\usepackage{color}
\usepackage{titlesec}
\newcommand{\bqq}{\begin{equation}}
\newcommand{\eqq}{\end{equation}}
\newcommand{\bqs}{\begin{equation*}}
\newcommand{\eqs}{\end{equation*}}

\newcommand{\R}{\mathbb{R}}

\newcommand{\cL}{\mathcal{L}}

\newtheorem{lem}{Lemma}[section]
\newtheorem{thm}{Theorem}
\newtheorem{prop}[lem]{Proposition}

\newtheorem{rmk}[lem]{Remark}

\numberwithin{equation}{section}

\author[1]{Gr\'egory Faye\footnote{ \texttt{gregory.faye@math.univ-toulouse.fr}}}
\author[2]{Thomas Giletti}
\author[3]{Matt Holzer}
\affil[1]{\small Institut de Math\'ematiques de Toulouse ; UMR 5219, Universit\'e de Toulouse ; UPS IMT, F-31062 Toulouse Cedex 9, France}
\affil[2]{\small IECL ; UMR 7502, University of Lorraine ; B.P. 70239, F-54506 Vandoeuvre-l\`es-Nancy Cedex, France}
\affil[3]{\small Department of Mathematical Sciences and Center for Mathematics and Artificial Intelligence (CMAI), George Mason University, Fairfax, VA, USA}

\begin{document}
\title{Asymptotic spreading for Fisher-KPP reaction-diffusion equations with heterogeneous shifting diffusivity}
\maketitle

\begin{abstract}
We determine the asymptotic spreading speed of the solutions of a Fisher-KPP reaction-diffusion equation, starting from compactly supported initial data, when the diffusion coefficient is a fixed bounded monotone profile that is shifted at a given forcing speed and satisfies a general uniform ellipticity condition. Depending on the monotony of the profile, we are able to characterize this spreading speed as a function of the forcing speed and the two linear spreading speeds associated to the asymptotic problems. Most notably, when the profile of the coefficient diffusion is increasing we show that there is an intermediate range for the forcing speed where spreading actually occurs at a speed which is larger than the linear speed associated with the homogeneous state around the position of the front.
We complement our study with the construction of strictly monotone traveling front solutions with strong exponential decay near the unstable state when the profile of the coefficient diffusion is decreasing and in the regime where the forcing speed is precisely the selected spreading speed.
\end{abstract}



\section{Introduction}

In this work we aim at understanding the influence of a heterogeneous diffusion on the spreading speed of solutions of a scalar reaction-diffusion equation of Fisher-KPP type \cite{fisher,KPP}. More precisely, we consider the situation of a shifting environment, i.e. the diffusion coefficient depends on a shifting variable $x - c_{het} t$ where the parameter $c_{het} \in \R$. The equation we consider writes as follows:
\begin{equation}\label{eq:main}
\partial_t u =  \chi ( x - c_{het} t) \partial_x^2 u + \alpha u (1-u), \quad t > 0, \ x \in \R,
\end{equation}
where $\alpha>0$. Throughout this paper we assume that the function $\chi: \R \mapsto \R$ is positive and smooth; further assumptions will be added below.

Such a reaction-diffusion equation is often used in ecology and population dynamics to model spatial propagation phenomena \cite{AW78}. In this context, the unknown function $u$ denotes the density of some species, and the parameter $\alpha$ stands for its per capita growth rate. The function $\chi$ stands for the motility of individuals, and therefore the heterogeneity means that the motility depends on local environmental conditions. In the past decade, there has been several series of works dedicated to the study of a variation of~\eqref{eq:main} with constant motility (say~$\chi \equiv 1$ to fix the ideas) but with shifting growth rates, that typically reads
\bqs
\partial_t u =  \partial_x^2 u+ u(r(x-c_{het}t)-u),
\eqs
where the function $r$ is not necessarily monotone and can take negative values \cite{BR08,berestycki2009can,bouhours2019spreading,li14,BF18}. One of the main motivations behind the introduction of a shifting growth rate came from modeling climate change and its impact on the survival of species \cite{berestycki2009can,cosner}. Some spreading-vanishing dichotomy properties were typically found, depending on either the size of a favorable patch~\cite{berestycki2009can} or of the initial population~\cite{bouhours2019spreading}. 

\medskip

Here our choice of a shifting heterogeneity (i.e. depending only on a shifting variable $x - c_{het} t$) comes from systems coupling several species densities. Most of the mathematical literature on spreading speeds for systems has been concerned with situations where these species only interact through their growth rates, e.g. prey-predator and competition systems \cite{dgm_sys,GL19,HS14,lam21}. Yet we would like to understand the influence that such inter-species interactions may play on the respective motilities of each species. For instance, it seems natural that the motility of predators should be higher when the density of the prey is low, while the motility of preys should be higher when the density of predators is high. 

To illustrate this, let us consider the reaction-diffusion system,
\begin{equation}\label{eq:sys0}
\left\{\begin{array}{l}
\partial_t u =  d(v) \partial_x^2 u + \alpha u (1-u),\vspace{3pt}\\
\partial_t v = \partial_x^2 v + \beta v (1-v),
\end{array}
\right.
\end{equation}
where $\beta >0$ and $d : \R^+ \mapsto \R$ is a positive function. Here $u$ and $v$ denote two distinct species, which are coupled through the motility of individuals of the former species. For simplicity we have excluded other coupling terms, and in particular there are no predation nor competition terms in the growth rates; this allows us to focus purely on the cross-diffusion phenomenon. As suggested above, the function $d$ may be taken as a monotonic function depending on whichever species acts as the predator. More precisely, if $u$ is the prey then the motility function $d$ should be increasing as the number of predators $v$ grows; conversely, if $u$ is the predator then its motility $d$ should be decreasing as the number of preys $v$ grows. Other forms of cross-diffusion have been proposed and studied in the literature, typically when the diffusion term is either of the form $\partial_x^2(d(v)u)$ or $\partial_x(d(v)\partial_xu)$ \cite{SKT,lewis,gatenby1996reaction}. While such systems have been considered for instance in the context of chemotaxis, little is known from the
point of view of propagation phenomena, apart from perturbative or singular limit results that reduce the system to a weakly coupled system which allow the construction of traveling front solutions (see for example \cite{gallay} in the context of cancer modeling). 

It is well-known that, for large classes of initial data, the solution of the Fisher-KPP equation
$$\partial_t v = \partial_x^2 v + \beta v (1-v),$$
converges to a traveling front solution, i.e. an entire in time solution $v(t,x) = V(x-ct)$ where $V$ is a decreasing function and satisfies
$$V(-\infty) = 1 \ , \quad V (+\infty) = 0.$$
More precisely, such a traveling front solution exists if and only if $c \geq 2 \sqrt{\beta}$, and the traveling front with minimal speed is typically the most biologically meaningful because it is attractive with respect to compactly supported initial conditions (up to some drift phenomenon) \cite{fisher,KPP,AW78}. Due to the unilateral coupling in system~\eqref{eq:sys0}, it is thus natural to assume that $v$ is identical to the traveling front with minimal speed, and then the equation for $u$ reduces to \eqref{eq:main} with $c_{het} = 2 \sqrt{\beta}$. In other words, the shifting heterogeneity arises as a result of the propagation of another surrounding species, which individuals of the species $u$ may either chase or flee from.
\begin{rmk}
From the above discussion, it may seem natural to consider a slightly more general heterogeneous diffusion $\widetilde{\chi} (t,x)$, which converges as $t \to +\infty$ to $\chi (x - c_{het} t)$, or even to $\chi (x - c_{het} t + m(t))$ where $m (t) = o (t)$ as $t\to +\infty$ accounts for the drift phenomenon. One may check that most of our analysis applies to such situations with straightforward adaptations. We choose to only consider \eqref{eq:main} since it makes the presentation simpler and captures the same phenomena.
\end{rmk}

\begin{figure}[!t]
\centering
\includegraphics[width=0.4\textwidth]{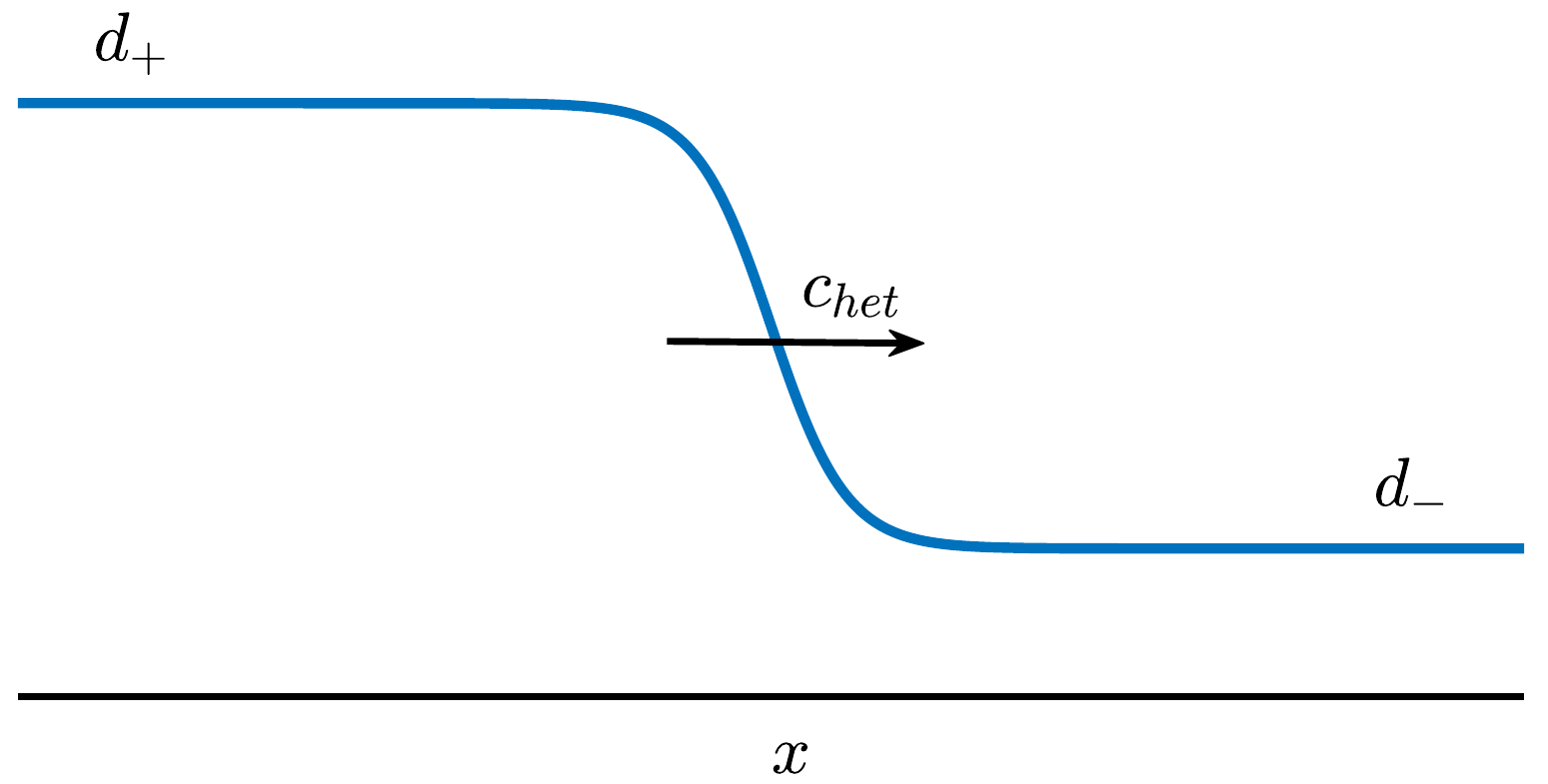}\hspace{1.5cm}
\includegraphics[width=0.4\textwidth]{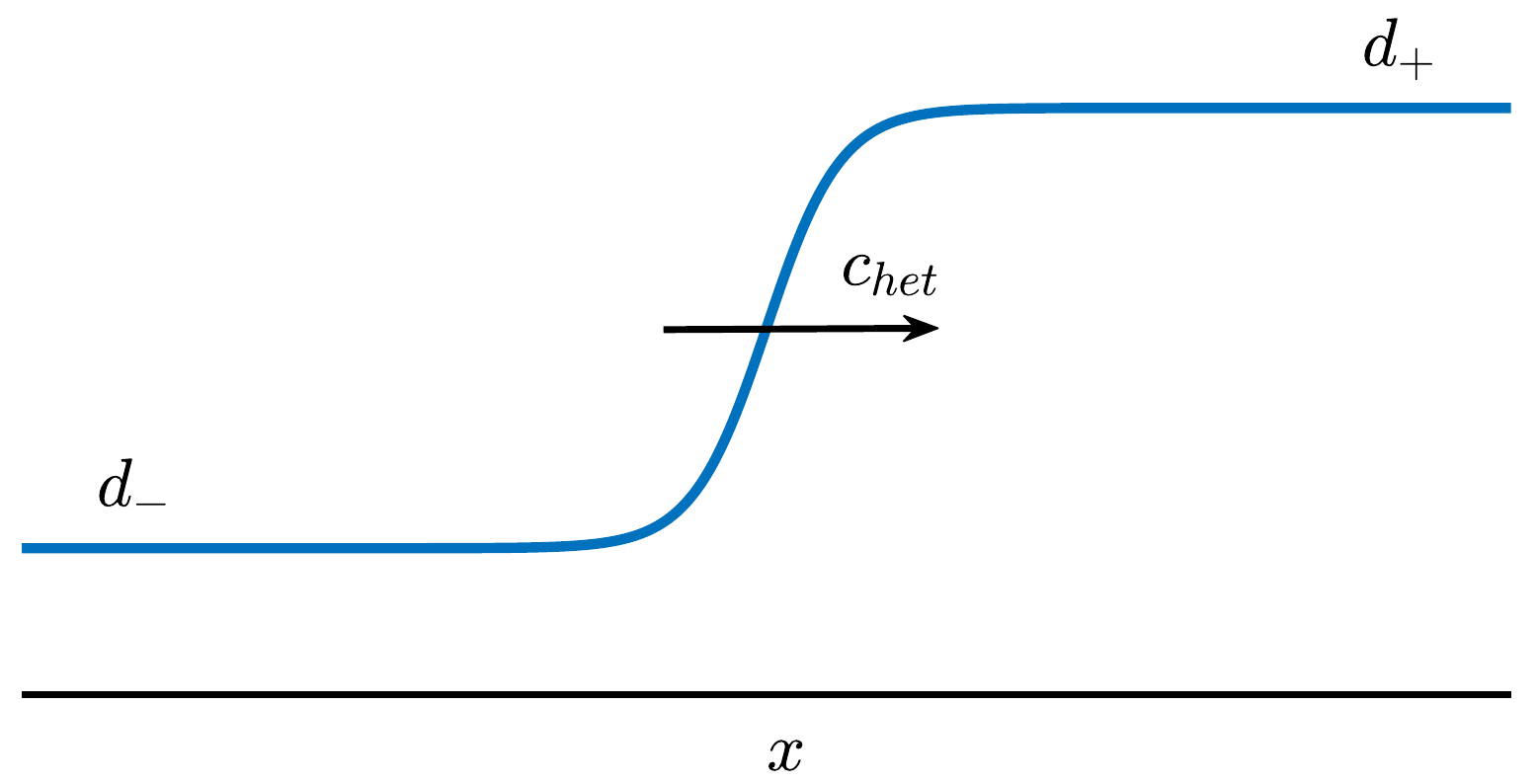}
\caption{Illustration of $\chi$ in cases~\eqref{ass:chi_dec} (left) and \eqref{ass:chi_inc} (right).}
\label{fig:Chi}
\end{figure}

Following this motivation, we will study the large-time behavior of solutions of~\eqref{eq:main} with $c_{het} \geq 0$ and under either assumptions
\begin{equation}\tag{I}\label{ass:chi_dec}
 \left\{\begin{array}{l}
\chi >0, \ \  \chi ' \leq 0, \vspace{3pt}\\
\chi (-\infty) = d_+ , \ \ \chi (+\infty) = d_-,\\
\end{array}
\right.
\end{equation}
or
\begin{equation}\tag{II}\label{ass:chi_inc}
 \left\{\begin{array}{l}
\chi >0, \ \  \chi ' \geq  0, \vspace{3pt}\\
\chi (-\infty) = d_- , \ \ \chi (+\infty) = d_+ ,\\
\end{array}
\right.
\end{equation}
where 
$$0 < d_- < d_+ < +\infty.$$

Let us point out that, equivalently, one may fix the function $\chi$ satisfying either assumption, allow $c_{het}$ to vary on the whole real line $\R$ and investigate spreading in both directions. 

Finally, equation~\eqref{eq:main} will be supplemented together with some initial condition $u_0$ such that
$$0 \leq u_0 \leq 1 \mbox{ and it is nontrivial and compactly supported}.$$

\section{Main results}

In the first case when \eqref{ass:chi_dec} holds, then the `favorable' part of the environment where individuals diffuse more is growing with speed $c_{het}$, and therefore one may expect that the solution will be spreading with positive speed. On the other hand, in the second case when \eqref{ass:chi_inc} holds, then the `favorable' part of the environment is actually receding and a key point will be whether the solution is able to `keep up' with the shifting speed $c_{het}$ of the heterogeneity.

This leads us to introduce, before we state our main results, the speeds
$$c_+ := 2 \sqrt{d_+ \alpha}, \quad c_- := 2 \sqrt{d_- \alpha},$$
which correspond respectively to the speeds of the homogeneous Fisher-KPP equation
$$\partial_t u = d \partial_x^2 u + \alpha u (1-u),$$
with either $d= d_+$ or $d=d_-$. In other words, these are the speeds respectively in the `favorable' environment where the motility function $\chi$ is at its maximal value, and in the `unfavorable' environment where $\chi$ is at its minimal value.. The large time behavior of solutions of \eqref{eq:main} will largely depend on how the shifting speed $c_{het}$ of the heterogeneity compares with the values of $c_-$ and $c_+$.

\subsection{Case \eqref{ass:chi_dec}}

We first consider the situation when the diffusivity is `high' (resp. `low') behind (resp. beyond) the moving frame with speed $c_{het}$. In this case, we will prove the following result:
\begin{thm}\label{theo:chi_dec}
Assume that $\chi$ satisfies case \eqref{ass:chi_dec}. Consider a compactly supported and nontrivial initial datum $u_0$ such that $0 \leq u_0 \leq 1$. Then the solution of \eqref{eq:main} spreads to the right with some positive speed $c^*_u$ in the sense that
$$\forall 0 < c < c^*_u, \quad \limsup_{t \to +\infty} \ \sup_{0 \leq x \leq ct} |1 - u(t,x) | =0,$$
$$\forall c > c^*_u, \quad \limsup_{t \to +\infty} \ \sup_{ x \geq ct} u(t,x) =0.$$
Moreover, we have either:
$$c^*_u = c_-  \quad  \mbox{ if } \ \  c_{het}<c_-,$$
or
$$c^*_u = c_{het} \quad \mbox{ if } \ \ c_- \leq c_{het} \leq c_+,$$
or
$$c^*_u = c_+ \quad  \mbox{ if } \ \  c_+ < c_{het}.$$
\end{thm}
It is indeed natural that the spreading speed $c^*_u$ (if it exists) should be less than both~$c_+$ (which is the speed when diffusivity is maximal) and the maximum of~$c_{het}$ and $c_-$ (since the solution may not spread faster than  $c_-$ in the part where diffusivity is minimal beyond $x \approx c_{het} t$). This theorem states that these intuitive upper bounds give precisely the spreading speed. We refer to Figure~\ref{fig:Cases} (a) for a numerical illustration.

\begin{figure}[!t]
\centering
 \subfigure[Case \eqref{ass:chi_dec}.]{\includegraphics[width=0.41\textwidth]{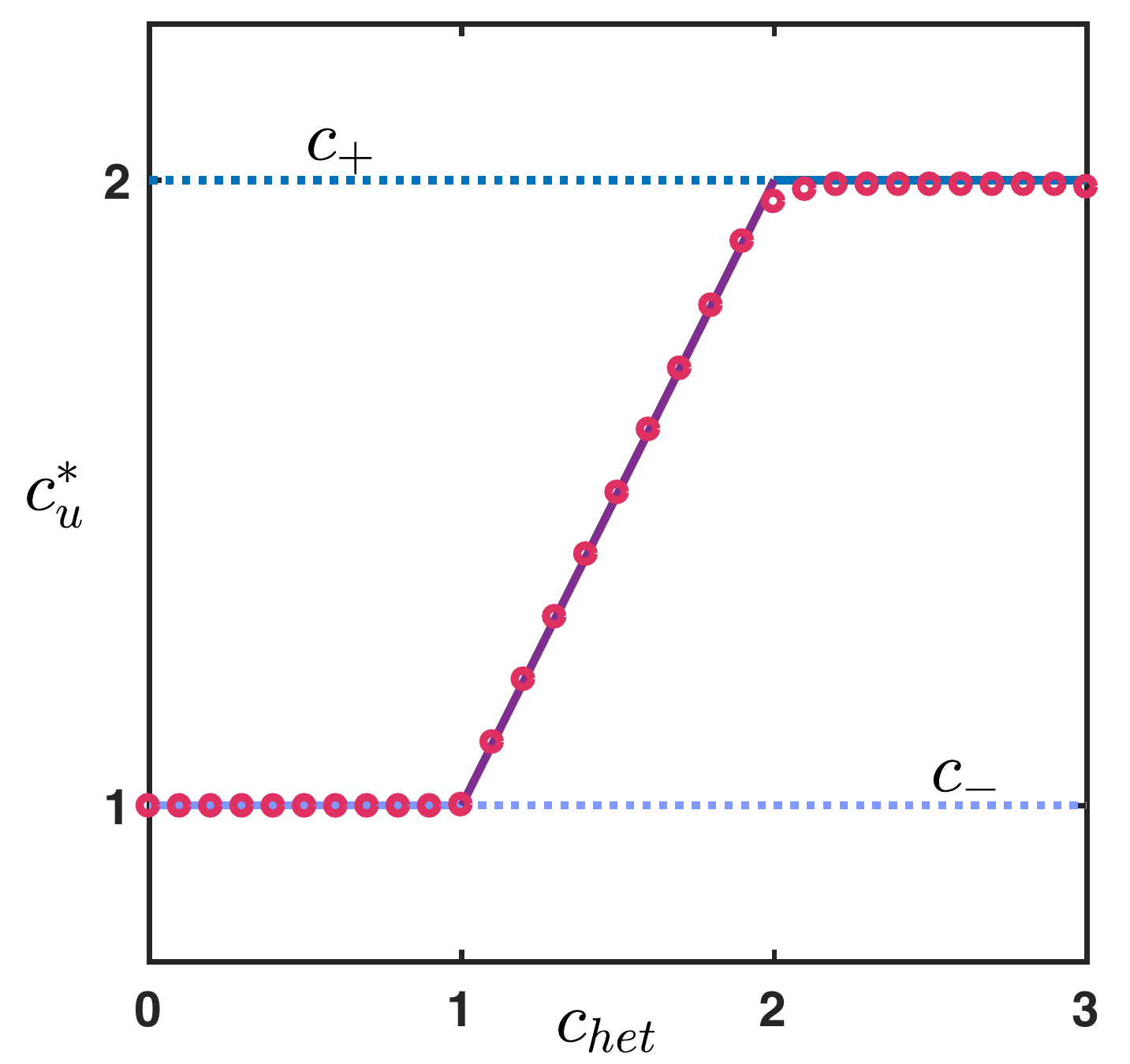}}
\hspace{1cm}
 \subfigure[Case \eqref{ass:chi_inc}.]{\includegraphics[width=0.49\textwidth]{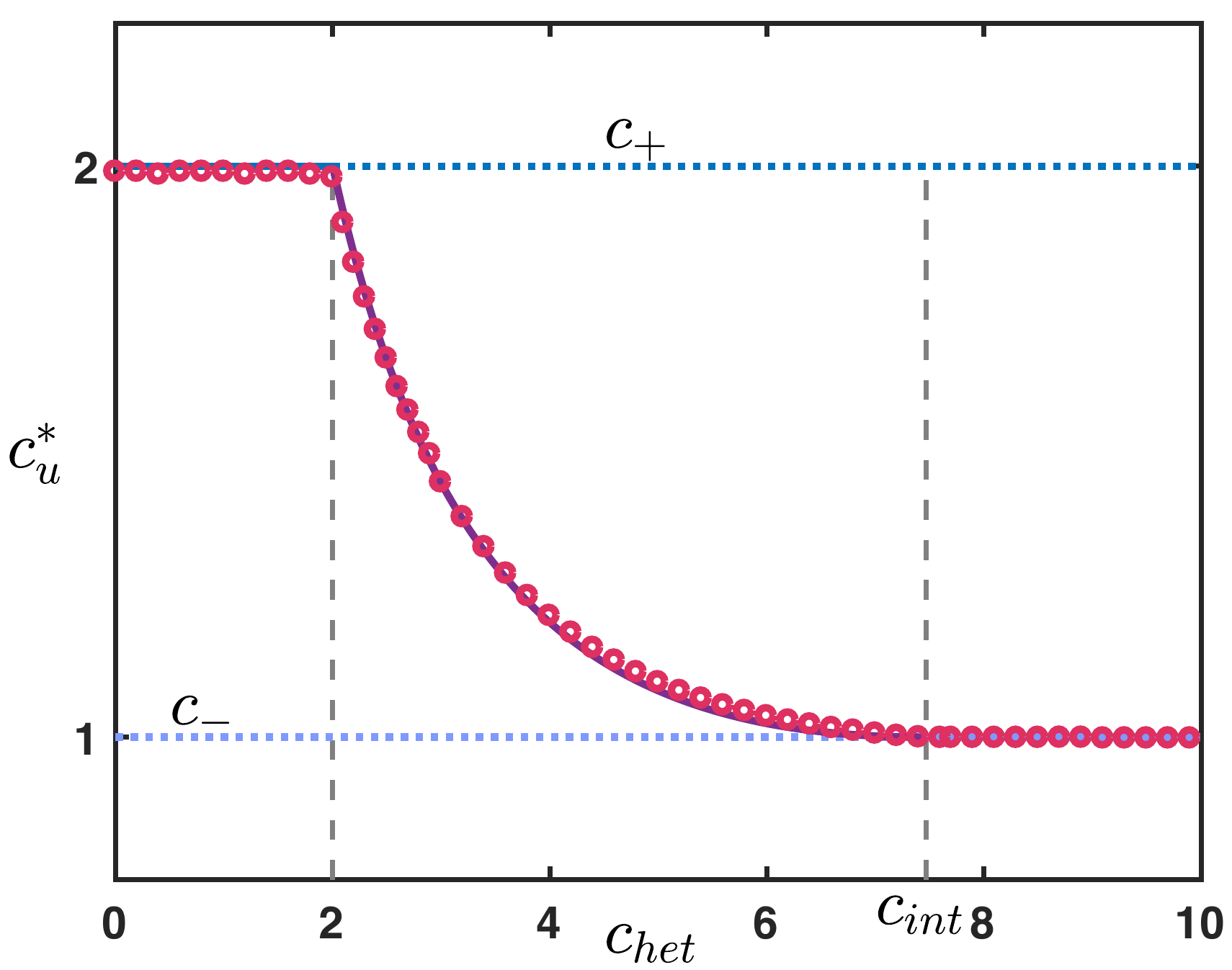}}
\caption{Numerically computed spreading speed $c_u^*$ (pink circles) as a function of $c_{het}$ for Case \eqref{ass:chi_dec} (left) and Case \eqref{ass:chi_inc} (right). The purple plain line is the theoretical spreading speed provided by Theorem~\ref{theo:chi_dec} and Theorem~\ref{theo:chi_inc}. In both cases parameters are fixed with $\alpha =1$, $d_+=1$ and $d_-=1/4$, such that the corresponding linear speeds are $c_+=2$ and $c_-=1$. The function $\chi$  was set to $\chi(x)=\frac{d_+ e^{-\lambda x} + d_-}{1+e^{-\lambda x}}$ in Case \eqref{ass:chi_dec} and to $\chi(x)=\frac{d_- e^{-\lambda x} + d_+}{1+e^{-\lambda x}}$ in Case \eqref{ass:chi_inc} with $\lambda=2$.}
\label{fig:Cases}
\end{figure}

We also note that when $c_{het}\in[c_-,c_+]$, the selected spreading speed $c_u^*$ is precisely the speed $c_{het}$ of the heterogeneity  and the invasion process for the $u$ component is locked to the heterogeneity; see also \cite{fh19}. We expect that in that case, the dynamics is dictated by traveling wave solutions~$U$ of
\begin{equation}
0=\chi(x) U''(x)+c_{het} U'(x)+\alpha U(x)(1-U(x)),\quad x\in\R,
\label{TWeq}
\end{equation}
that satisfy the conditions
\begin{equation}
U(-\infty)=1, \quad U(+\infty)=0, \quad \text{ with } \quad 0<U<1.
\label{TWlim}
\end{equation}

In order to state our second main result, we need to further assume the following assumption
\bqq
\left|\chi(x)-d_{\pm}\right| = \mathcal{O}\left(e^{-\nu|x|}\right), \quad \left|\chi'(x)\right| = \mathcal{O}\left(e^{-\nu|x|}\right), \quad  \text{ as } x\rightarrow\pm\infty,
\label{exp_decay}
\eqq
for some $\nu>0$. Although technical, the above assumption is natural if we come back to the reaction-diffusion system \eqref{eq:sys0} for which the traveling front solutions $v(t,x)=V(x-ct)$ for $c\geq 2\sqrt{\beta}$ are known to converge at an exponential rate towards their asymptotic limit states.

\begin{thm}\label{thmTF}
Assume that $\chi$ satisfies case \eqref{ass:chi_dec} and that \eqref{exp_decay} is verified. For each $c_{het}\in(c_-,c_+)$, there exists a unique strictly monotone traveling wave solution~$U$ of \eqref{TWeq}-\eqref{TWlim} with strong exponential decay at $+\infty$:
\begin{equation*}
U (x) \underset{x \rightarrow +\infty}{\sim} \gamma_s e^{-\lambda_s x}, \quad \lambda_s:=\frac{c_{het}+\sqrt{c_{het}^2-c_-^2}}{2d_-},
\end{equation*}
for some $\gamma_s>0$. For $c_{het}\in[0,c_-)$ or $c_{het}>c_+$ no such traveling wave can exist.
\end{thm}
We conjecture that $u(t,x)$ the solution of \eqref{eq:main} converges to the traveling wave~$U$ as $t \to +\infty$ in the moving frame with speed~$c_{het}$, but we do not address this issue here.

\subsection{Case \eqref{ass:chi_inc}}\label{sec:case_ass_chi_inc}

The other situation turns out to be slightly more intricate. Our main result writes as follows:
\begin{thm}\label{theo:chi_inc}
Assume that $\chi$ satisfies case \eqref{ass:chi_inc} and that \eqref{exp_decay} is verified. Consider a compactly supported and nontrivial initial datum $u_0$ such that $0 \leq u_0 \leq 1$. Then the solution of \eqref{eq:main} spreads to the right with some speed $c^*_u$ in the sense that
$$\forall 0 < c < c^*_u, \quad \limsup_{t \to +\infty} \ \sup_{0 \leq x \leq ct} |1 - u(t,x) | =0,$$
$$\forall c > c^*_u, \quad \limsup_{t \to +\infty} \ \sup_{ x \geq ct} u(t,x) =0.$$
Moreover, we have
$$c^*_u = c_+  \quad  \mbox{ if } \ \ c_{het} < c_+ ,$$
while
$$c^*_u =\frac{c_{het}}{2}\left(1 - \sqrt{1-\frac{d_-}{d_+}}\right)+\frac{c_-^2}{2 c_{het} \left(1 - \sqrt{1-\frac{d_-}{d_+}}\right)} \in (c_-, c_+]  \quad \mbox{ if } \ \  c_+ \leq c_{het}<  c_{int},$$
and
$$c^*_u = c_-  \quad \mbox{ if } \ \  c_{int} < c_{het} .$$
Here 
$$c_{int} :=c_+ \left(  \sqrt{\frac{d_+}{d_-}} + \sqrt{ \frac{d_+}{d_-} - 1 } \right).$$
\end{thm}

Again we find several subcases depending on the value of $c_{het}$. When $c_+ > c_{het}$, then individuals move fast enough to keep up with the favorable zone beyond the front of the heterogeneity, so that propagation reaches it full speed $c_+$. In particular, the population spreads as if there was no heterogeneity.

The case when $c_+ < c_{het}$ is less intuitive. Because the solution cannot spread faster than $c_+$ and therefore it vanishes in the `favorable zone' as times goes to infinity, one may have expected that the solution behaves as in the equation with lower diffusivity $d_-$ and thus spreads with speed $c_-$. According to Theorem~\ref{theo:chi_inc}, it turns out that this intuition is true if $c_{het}$ is large enough. Yet there exists some intermediate range where spreading actually occurs with speed $c_u^*$ which is strictly larger than $c_-$. This means that the far away `favorable' zone still plays a role in the propagation, which can be related to the phenomena of accelerated fronts  or nonlocal pulling \cite{GL19,HS14}. We refer to Figure~\ref{fig:Cases} (b) for a numerical illustration.

This observation is especially striking when $d_- $ is small. Indeed, let us take the formal limit $d_- = 0$ in Theorem~\ref{theo:chi_inc}. Then $c_{int}= c_+ \left(  \sqrt{\frac{d_+}{d_-}} + \sqrt{ \frac{d_+}{d_-} - 1 } \right)  \to +\infty$ and it can also be checked that $c_u^* \to \frac{4 d_+ \alpha}{c_{het}} >0 $ for any $c_{het} >c_+$ as $d_- \to 0$. This strikingly suggests that, even if diffusivity vanishes around any positive level set of the solution, spreading still occurs with a positive speed.\medskip

Before we proceed to the proofs, let us give some formal computation to explain the appearance of an anomalous speed. Assume that $c_+ < c_{het}$, and that
$$\chi (z) := \left\{ \begin{array}{l}
d_- \mbox{ if } z \leq 0, \vspace{3pt}\\
d_+ \mbox{ if } z  > 0.
\end{array}
\right.
$$
Notice that such $\chi$ is no longer smooth but it allows us to perform an explicit computation. Since the reaction term $\alpha u (1-u)$ is concave, it is of the Fisher-KPP type and it is reasonable to expect that the spreading speed is dictated by the linearized equation
$$\partial_t u = \chi (x-c_{het} t) \partial_x^2 u + \alpha u .$$
By analogy with the usual Fisher-KPP equation, it is also natural to look for an exponential ansatz. However, if one tries an ansatz of the type
$$e^{-\lambda (x - ct)},$$
one immediately sees that both resulting dispersion equations (for $x \leq c_{het} t$ and $x > c_{het} t$) cannot be satisfied simultaneously. Therefore, we instead look for an ansatz of the type
$$\left\{ \begin{array}{l}
e^{-\lambda (x-ct)} \mbox{ if } x \leq c_{het} t,\vspace{3pt}\\
e^{-\lambda (c_{het} -c) t} \times e^{-\mu (x - c_{het} t)} \mbox{ if } x \geq c_{het} t ,
\end{array}
\right.
$$
where $\lambda, \mu >0$ and $c \in \R$. Note that the resulting function is continuous. Plugging this in the linearized equation, we find that
$$d_- \lambda^2 - c\lambda + \alpha = 0 \quad \mbox{and} \quad  d_+ \mu^2 - c_{het} \mu + \lambda (c_{het} - c)  + \alpha = 0.$$
The resulting conditions for finding such an ansatz are
\begin{equation}\label{eq:anomalous_condition0}
c \geq c_-, \quad g (c) \geq 0,
\end{equation}
where
$$g(c) := c_{het}^2 - 4 d_+ \left[ \alpha  + \frac{c - \sqrt{c^2 - 4  d_- \alpha}}{2 d_-} (c_{het} - c) \right] .$$
On the one hand, it is easy to check that both conditions are satisfied when $c \geq c_{het}$ (recall that $c_{het} >c_+$ here). On the other hand,  it is straightforward to compute that 
$$\forall c \in [c_- , c_{het}], \quad g ' (c) > 0.$$
In particular, if $g (c_- ) \geq 0$, then we conclude that there is an exponential ansatz for any $c \geq c_- $. This occurs when
$$c_{het}^2 - 4 d_+ \sqrt{\frac{\alpha}{d_-}} c_{het} + 4 d_+ \alpha \geq 0,$$
which is equivalent to
$$c_{het} \geq c_+ \left(  \sqrt{\frac{d_+}{d_-}} + \sqrt{ \frac{d_+}{d_-} - 1 } \right)=c_{int}.$$
By analogy with the homogeneous case, it is reasonable to expect that under the previous condition the solution spreads with speed~$c_-$ which is precisely what is stated in Theorem~\ref{theo:chi_inc}. 

It remains to consider the case when $g (c_-) < 0$. Then we claim that
$$g (c_+ ) >0.$$
To check this, first compute
\begin{equation}\label{cla:g}
g(c_+) = c_{het}^2 - 4 d_+ \left[ \alpha  + \frac{c_+ - \sqrt{c_+^2 - 4 d_- \alpha }}{2 d_-} (c_{het} - c_+) \right].
\end{equation}
Notice that $g (c_+ ) = 0$ when $c_{het} = 0$. Hence it is enough to show that the derivative with respect to $c_{het}$ is positive, which is indeed the case as
\begin{eqnarray*}
2 c_{het} - 2 \frac{d_+}{d_-} \left[ c_+ - \sqrt{c_+^2 - 4 d_- \alpha } \right] 
& \geq & 2 c_+  - 2 \frac{d_+}{d_-} \left[ c_+ - \sqrt{c_+^2 - 4 d_- \alpha } \right] \\
& \geq &  4 \sqrt{d_+ \alpha} \left[ 1 - \frac{d_+}{d_-} + \frac{d_+}{d_-} \sqrt{1 - \frac{d_-}{d_+}} \right] \\
& \geq & 4 \sqrt{d_+ \alpha} \left[ 1 - \frac{d_-}{2 d_+}  \right] \\
& > & 0.
\end{eqnarray*}
Claim~\eqref{cla:g} is proved and it follows that $g^{-1} (0) \in (c_- , c_+)$ (here $g$ is understood as a function on the interval $[c^-, c^+]$ where it is invertible), and that an exponential ansatz exists if and only $c \geq g^{-1} (0)$. Furthermore, upon denoting $\lambda_\star := \alpha - \frac{c_{het}^2}{4d_+}<0$ (as $c_+<c_{het}$) and letting 
$$c= c_{het} - \frac{1}{\zeta}, $$
one may check that $g(c) = 0$ is equivalent to
$$2 d_- \lambda_\star \zeta^2 + c_{het} \zeta - 1 =  \sqrt{ (c_{het} \zeta - 1)^2 - 4 d_- \alpha \zeta^2  } .$$
The part inside the square root is positive, so that one eventually reaches
$$d_- \lambda_\star \zeta^2 +  c_{het} \zeta - 1  +  \frac{\alpha}{\lambda_\star}  = 0.$$
Since $\zeta$ must be positive in order for $c$ to belong to the interval $(c_-, c_+)$, we conclude that
$$g^{-1} (0) = c_{het}-\frac{2d_-\lambda_\star}{-c_{het}+\sqrt{c_{het}^2-4d_-(\alpha-\lambda_\star)}},$$
and we recover the formula for the spreading speed in Theorem~\ref{theo:chi_inc} by using the expression for $\lambda_\star$.

Building upon the argument above, we can further motivate the emergence of the accelerated front with speed $c_u^*$ as a matching condition in the following sense. We attempt to build a solution of~(\ref{eq:main}) which resembles a traveling front of the Fisher-KPP equation $0=d_- U''+cU'+\alpha U(1-U)$ on the left concatenated on the right with a solution of the linearized equation near zero in a frame moving with the heterogeneity.  If $c>c_-$ then there exists a family of traveling fronts traveling with speed $c$ and whose profile resembles $e^{-\lambda(c)(x-ct)}$ as $x-ct \to +\infty$.  In a frame of reference moving with speed $c_{het}$ this front can be viewed as forming an effective boundary condition for the PDE (\ref{eq:main}) linearized near zero.  Rescaling by $u(t,x):=e^{\eta(c) t}v(t,x)$ this PDE satisfies the boundary value problem 
\begin{equation} v_t=\chi(x) v_{xx}+c_{het}v_x+(\alpha-\eta(c)) v, \quad  v_x(-L)=-\lambda(c)v(-L), \quad v_x(L)=-\gamma(c)v(L) \label{eq:veqntry}  \end{equation}
for some $L$ sufficiently large and $\eta(c)=-\lambda(c) (c_{het}-c)$.  For generic separated boundary conditions, it is known that the point spectrum of the linear operator appearing on the right hand side of (\ref{eq:veqntry}) will accumulate on the absolute spectrum in the limit as $L\to\infty$; see \cite{sandstede00,kapitula}.   The absolute spectrum in this case is
\[ \Sigma_{abs}=\left\{ \lambda \leq \alpha-\eta(c) -\frac{c_{het}^2}{4d_+} \right\}. \]
Thus, if $c$ is selected so that $\eta(c)<\alpha -\frac{c_{het}^2}{4d_+}$ then one would expect that the constructed solution would be pointwise unstable due to unstable point spectrum of (\ref{eq:veqntry}).  Conversely, if $\eta(c)>\alpha-\frac{c_{het}^2}{4d_+}$ then solutions of (\ref{eq:veqntry}) will decay pointwise and matching with the front will fail.  Therefore, we select $c_u^*$ so that 
\[ \eta(c_u^*)=\alpha-\frac{c_{het}^2}{4d_+}=\lambda_\star, \]
which upon inspection is equivalent to the condition that $g(c_u^*)=0$.

\paragraph{Outline of the paper.} In Section~\ref{sec:gen}, we provide general bounds on the spreading speed, and prove that if it exists it should be bounded respectively from above and below by $c_+$ and $c_-$. Sections~\ref{sec:dec} and~\ref{sec:inc} are dedicated to the proofs of our main theorems on spreading speeds in cases \eqref{ass:chi_dec} and \eqref{ass:chi_inc} respectively. For each case, we construct sub and/or super-solutions to bound adequately the spreading speed. Finally, in the last Section~\ref{sec:TW}, we prove the existence and uniqueness of traveling front solutions with strong exponential decay at $+\infty$ in case \eqref{ass:chi_dec} and for each $c_{het}\in(c_-,c_+)$.

\section{General bounds on the spreading speed}\label{sec:gen}

In this section, we aim at confirming rigorously the natural intuition that the spreading speed should always be bounded from respectively above and below by $c_+ = 2 \sqrt{d_+ \alpha}$ and $c_- = 2 \sqrt{d_- \alpha}$. We will do this by constructing some super and sub-solutions which are valid in all the cases considered in our main theorems. Actually, our bounds on the spreading speed remain true for more general form of diffusivity and therefore our results here can be of independent interest. 

Throughout this section, we consider
\begin{equation}
\partial_t u =  \chi (t,x) \partial_{x}^2 u + \alpha u (1-u),
\label{kpp}
\end{equation}
with $\alpha>0$ and some smooth function $\chi$ that only satisfies $$ d_- \leq \chi(t,x) \leq  d_+ \ \mbox{ for all } \ t\geq0 \text{ and } \ x\in\R ,$$
where 
$$0< d_- < d_+.$$
We will prove the following result.
\begin{prop}\label{prop:gen}
Consider a compactly supported and nontrivial initial datum $u_0$ such that $0 \leq u_0 \leq 1$. Then the solution $u$ of \eqref{kpp} satisfies that
$$\forall \, 0 < c < c_-, \quad \limsup_{t \to +\infty} \ \sup_{0 \leq x \leq ct} |1 - u(t,x) | =0,$$
$$\forall \, c > c_+, \quad \limsup_{t \to +\infty} \ \sup_{ x \geq ct} u(t,x) =0.$$
\end{prop}
In particular, the rightward spreading speed $c_u^*$ (if it exists) must satisfy $c_-\leq c_u^*\leq c_+$.
We refer to \cite{BN12,BN21} for other general results on monostable equations with spatio-temporal heterogeneities.

\subsection{A general super-solution}

We start the proof of Proposition~\ref{prop:gen} with the construction of a super-solution.
\begin{lem}\label{gensuper}
For any $C >0$, the function $w(t,x)=\min\left\{1,C e^{-\frac{c_+}{2d_+}(x-c_+t)}\right\}$ is a super-solution of~\eqref{kpp}.
\end{lem}

\begin{proof}
The proof is a simple calculation. Let 
\bqs
N(v):=\partial_t v- \chi (t,x) \partial_{x}^2 v - \alpha v(1-v).
\eqs
We compute $N(1)=0$ and, for any $C>0$,
\begin{align*}
N(C e^{-\lambda(x-ct)})&= C \left(\lambda c-d_+ \lambda^2-\alpha\right)e^{-\lambda(x-ct)}+ C \lambda^2(d_+ -\chi(t,x))e^{-\lambda(x-ct)}+\alpha C^2 e^{-2\lambda(x-ct)}\\
&> C \left(\lambda c-d_+ \lambda^2-\alpha\right)e^{-\lambda(x-ct)},
\end{align*}
as $d_+ -\chi(t,x) \geq 0$ for all $x\in\R$ and $t\geq 0$. Then letting $\lambda=\frac{c_+}{2d_+}$ and $c=c_+$ we get that $w(t,x)$ is a super-solution.
\end{proof}
Now, for any $0 \leq u_0 \leq 1$ a compactly supported and nontrival initial datum, one can find $C>0$ large enough so that $u_0 \leq \min \{ 1, C e^{-\frac{c_+}{2d_+} x } \}$. A consequence of the above lemma and the comparison principle is that we have
\bqs
\forall \, c > c_+, \quad \limsup_{t \to +\infty} \ \sup_{ x \geq ct} u(t,x) =0,
\eqs
for $u(t,x)$ solution of \eqref{kpp} from a compactly supported, nontrivial initial datum $0\leq u_0 \leq 1$. 
This proves the first assertion of Proposition~\ref{prop:gen}.

\subsection{A general sub-solution}

We now construct a sub-solution and conclude the proof of Proposition~\ref{prop:gen}. For each $c\in(0,c_-)$, there exist $\delta>0$ small enough and $\eta>0$ such that
\bqs
0<c<2\sqrt{d_-(\alpha-\delta)}<c_-<2\sqrt{d_+(\alpha-\delta)}<c_+,
\eqs 
and
\bqs
(\alpha-\delta)u<\alpha u(1-u), \quad \forall u \in(0,\eta).
\eqs

We first set
\bqs
\lambda := \frac{c}{2d_-}, \quad \text{ and } \quad \beta := \frac{1}{2d_-}\sqrt{4d_-(\alpha-\delta)-c^2}>0,
\eqs
together with
\bqs
\gamma := \frac{c}{2d_+}, \quad \text{ and } \quad \omega := \frac{1}{2d_+}\sqrt{4d_+(\alpha-\delta)-c^2}>0.
\eqs
Then, the functions $\Psi_-(z):=e^{-\lambda z}\cos(\beta z)$ and $\Psi_+(z):=e^{-\gamma z}\cos(\omega z)$ are respectively solutions of
\bqs
d_- \Psi'' +c\Psi'+(\alpha-\delta)\Psi=0,
\eqs
and
\bqs
d_+ \Psi'' +c\Psi'+(\alpha-\delta)\Psi=0.
\eqs
We consider the intervals $\Omega_-=\left(-\frac{\pi}{2\beta},\frac{\pi}{2\beta}\right)$ and $\Omega_+=\left(-\frac{\pi}{2\omega},\frac{\pi}{2\omega}\right)$, and we note that $\Psi_\pm > 0$ on $\Omega_\pm$. We denote $z^*_\pm \in \Omega_\pm$ the unique values where $\Psi''_\pm(z^*_\pm)=0$, which are given by
\bqs
z^*_-=-\frac{1}{\beta}\mathrm{arctan}\left( \frac{\lambda^2-\beta^2}{2\lambda \beta}\right), \quad \text{ and } \quad z^*_+=-\frac{1}{\omega}\mathrm{arctan}\left( \frac{\gamma^2-\omega^2}{2\gamma \omega}\right).
\eqs

\begin{figure}[!t]
\centering
\includegraphics[width=0.49\textwidth]{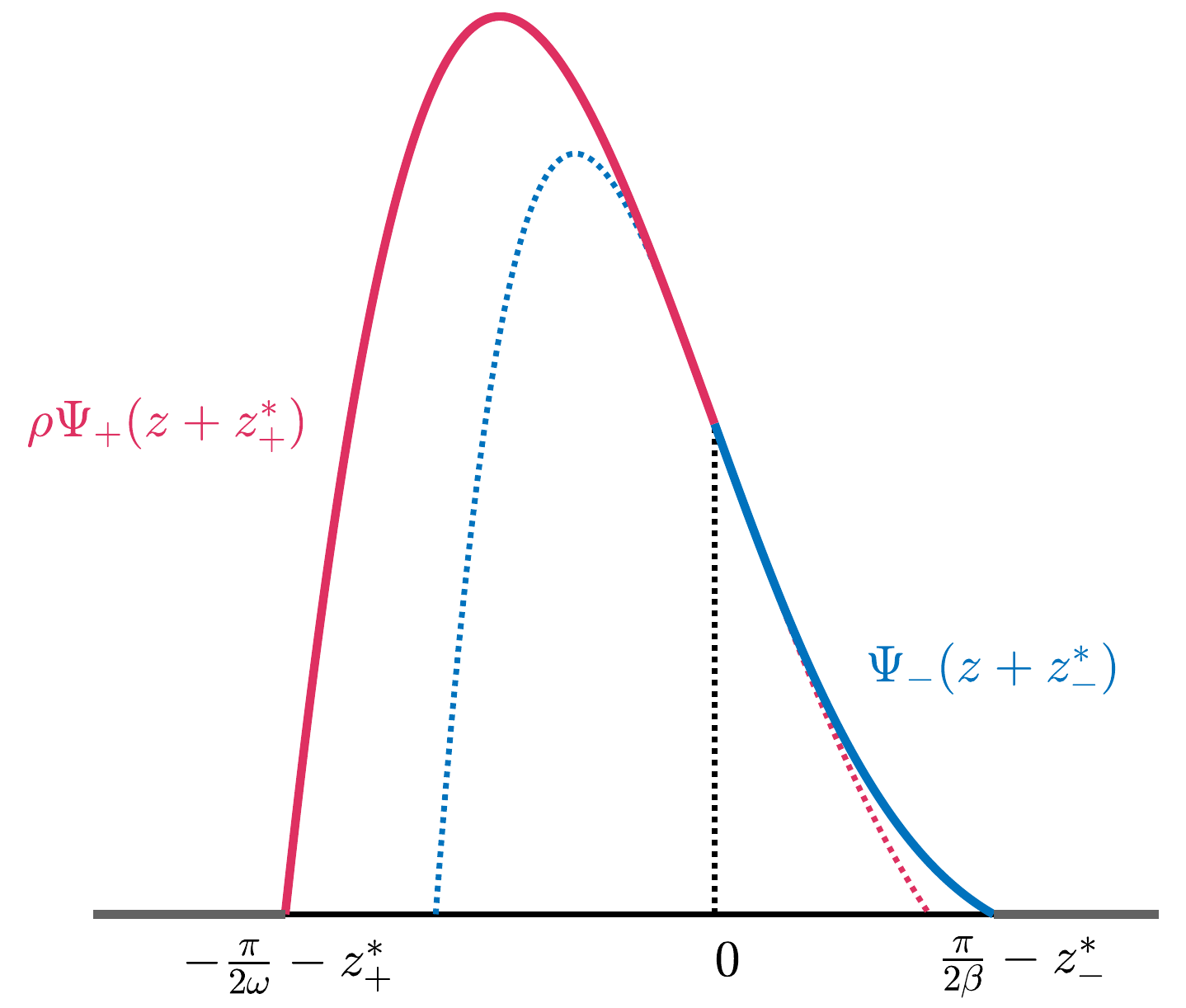}
\caption{Illustration of the building block of the general sub-solution \eqref{gensubsoleq} (before its scaling by~$\epsilon$) which is composed of two parts $\rho \Psi_+$ (pink curve) and $\Psi_-$ (blue curve) in the moving frame $z=x-ct$. It is of class $\mathscr{C}^2$ and compactly supported on $\left[-\frac{\pi}{2\omega}-z_+^*,\frac{\pi}{2\beta}-z_-^*\right]$.}
\label{fig:GenSubSol}
\end{figure}

Finally, we introduce the function
\bqq
\underline{u}_{c,\delta,\epsilon}(t,x)=\left\{
\begin{array}{lc}
0,& x-ct \leq -\frac{\pi}{2\omega}-z_+^*,\\
\epsilon \rho \Psi_+(x-ct+z_+^*), & -\frac{\pi}{2\omega}-z_+^*<x-ct\leq 0,\\
\epsilon \Psi_-(x-ct+z_-^*), & 0<x-ct<\frac{\pi}{2\beta}-z_-^*,\\
0,& x-ct\geq \frac{\pi}{2\beta}-z_-^*.
\end{array}
\right.
\label{gensubsoleq}
\eqq
Here, $\rho>0$ is chosen so as to ensure continuity at $x-ct=0$, that is
\bqs
\rho:=\frac{\Psi_-(z_-^*)}{\Psi_+(z_+^*)}.
\eqs
We fix $\epsilon>0$ small enough such that 
\bqs
0 \leq \underline{u}_{c,\delta,\epsilon}(t,x) < \eta, \quad \forall t\geq0, \, x\in\R.
\eqs
We now check that $\underline{u}_{c,\delta,\epsilon}(t,x)$ is a generalized sub-solution.

\begin{itemize}
\item For $z=x-ct\in\left(-\frac{\pi}{2\omega}-z_+^*,0\right] $, one has $0<\epsilon \rho \Psi_+(z+z_+^*)< \eta$, and so 
\bqs
N(\epsilon \rho \Psi_+(z+z_+^*))<(d_+-\chi(t,x))\epsilon \rho \Psi_+''(z+z_+^*)\leq 0,
\eqs
as $\Psi_+''(z+z_+^*)\leq 0$ on $\left(-\frac{\pi}{2\omega}-z_+^*,0\right]$ and $\chi(t,x)\leq d_+$.
\item For $z=x-ct\in\left(0,\frac{\pi}{2\beta}-z_-^*\right)$, one has 
$0<\epsilon  \Psi_-(z+z_-^*)< \eta$, and so 
\bqs
N(\epsilon \Psi_-(z+z_-^*))<(d_--\chi(t,x))\epsilon \rho \Psi_-''(z+z_-^*)\leq 0,
\eqs
as $\Psi_-''(z+z_-^*)\geq 0$ on $\left(0,\frac{\pi}{2\beta}-z_-^*\right)$ and $d_-\leq \chi(t,x)$.
\item At $z=x-ct=0$, we already have continuity since $\rho\Psi_+(z_+^*)=\Psi_-(z_-^*)$, and we also have that $\Psi_+''(z_+^*)=\Psi_-''(z_-^*)=0$ by definition of $z_\pm^*$. Now using the equations satisfied by $\Psi_\pm$ evaluated at $z_\pm^*$, we also obtain
\bqs
c\left( \Psi_-'(z_-^*)-\rho\Psi_+'(z_+^*)\right)=-(\alpha -\delta)\left( \Psi_-(z_-^*)-\rho\Psi_+(z_+^*)\right)-d_-\Psi_-''(z_-^*)+d_+\rho\Psi_+''(z_+^*)=0.
\eqs
As a consequence, we have $\Psi_-'(z_-^*)=\rho\Psi_+'(z_+^*)$ and the sub-solution is of class $\mathscr{C}^2$ for $x-ct\in \left(-\frac{\pi}{2\omega}-z_+^*,\frac{\pi}{2\beta}-z_-^* \right)$.
\end{itemize}
We have now reached the next result.
\begin{lem}\label{gensub}
Let $c\in(0,c_-)$. Then there exists $\delta_0(c)>0$ such that for each $0<\delta<\delta_0(c)$ one can find $\epsilon_0(c,\delta)>0$  such that for all $0<\epsilon<\epsilon_0(c,\delta)$ the function
\bqs
\underline{u}_{c,\delta,\epsilon}(t,x)=\left\{
\begin{array}{lc}
0,& x-ct \leq -\frac{\pi}{2\omega}-z_+^*,\\
\epsilon \rho \Psi_+(x-ct+z_+^*), & -\frac{\pi}{2\omega}-z_+^*<x-ct\leq 0,\\
\epsilon \Psi_-(x-ct+z_-^*), & 0<x-ct<\frac{\pi}{2\beta}-z_-^*,\\
0,& x-ct\geq \frac{\pi}{2\beta}-z_-^*,
\end{array}
\right.
\eqs
is a generalized sub-solution where $\Psi_\pm$, $z_\pm^*$ and $\rho$ are defined as above. We refer to Figure~\ref{fig:GenSubSol} for an illustration.
\end{lem}
With the above lemma, we can now conclude the proof of Proposition~\ref{prop:gen}. Let $u$ be the solution of the Cauchy problem starting from a compactly supported initial condition $0\leq u_0 \leq 1$, and fix $0 < c < c' < c_-$. By the strong maximum principle, it satisfies $u(1,\cdot)>0$ and, upon diminishing further $\epsilon$ in the previous lemma, one can ensure that 
\bqs
\underline{u}_{c',\delta,\epsilon}(1,\cdot) \leq u(1,\cdot) .
\eqs
By the comparison principle, we infer that
$$\inf_{t >1 } u(t,c' t) \geq \nu >0,$$
for some $\nu \in (0,1)$ which depends on~$c'$. By a straightforward symmetry argument, one may check that 
$\underline{u}_{c',\delta,\epsilon} (t,-x)$ is also a sub-solution of \eqref{kpp} and therefore we also get that
$$\inf_{ t > 1 } u(t,-c' t) \geq \nu >0.$$
Next, up to reducing $\nu$ we can also assume that
$$u(1,x) \geq \nu,$$
for all $x \in [-c',c']$. Therefore, applying a comparison principle on $\{ (t,x) \, | \ t \geq 1 \mbox{ and } x \in [-c' t,c't]\}$, with the constant $\nu$ as a sub-solution, we infer that actually
$$\liminf_{ t\to +\infty} \inf_{-c' t  \leq x \leq c' t} u (t,x) \geq \nu.$$
Now we prove the second assertion of Proposition~\ref{prop:gen}. We proceed by contradiction and, since $u \leq 1$, we assume that there exist sequences $t_n \to +\infty$ and $x_n \in [0 , ct_n]$ such that
$$\limsup_{n \to +\infty} u(t_n, x_n) < 1.$$
Up to extraction of a subsequence and by standard parabolic estimates, we have that the function $u(t_n+ t, x_n +x)$ converges as $n \to +\infty$ to an entire in time solution $u_\infty$ of
$$\partial_t u_\infty = \widetilde{\chi} (t,x) \partial_x^2 u_\infty + \alpha u_\infty (1 - u_\infty),$$
for some $\widetilde{\chi}$ which also satisfies $d_-\leq \widetilde{\chi} \leq d_+$. Moreover, it follows from the above where $c' >c$ that $u_\infty (t,x) \geq \nu>0$ for all $(t,x) \in \R^2$. By a straightforward comparison with the ODE and regardless of the actual function $\widetilde{\chi}$, the function~$u_\infty$ must be identical to 1, which is a contradiction with the fact that $u_\infty( 0,0) < 1$ by construction. This proves the last statement of Proposition~\ref{prop:gen}.

\section{Construction of super and sub-solutions for case \eqref{ass:chi_dec} and proof of Theorem~\ref{theo:chi_dec}}\label{sec:dec}

In this section, we construct super and sub-solutions in case where $\chi$ satisfies assumption \eqref{ass:chi_dec}. We proceed step by step and consider each case depicted in Theorem~\ref{theo:chi_dec}.

\subsection{Subcase $c_{het}<c_-$}

When $c_{het}<c_-$, we need to prove that the rightward spreading speed is $c_u^*=c_-$. From Proposition~\ref{prop:gen}, we already have proved that 
$$\forall 0 < c <c_-, \quad \limsup_{t \to +\infty} \sup_{0 \leq x \leq ct} |1 - u (t,x)| = 0,$$
and thus it remains to provide a super-solution in this case.

We let $c>c_-$. Then, one can find $\epsilon>0$ small enough such that we have $c_-=2 \sqrt{d_- \alpha}<2 \sqrt{(d_-+\epsilon) \alpha}< \min(c,c_+)$. Next we fix $\tau_\epsilon>0$ to be large enough such that for any $\tau \geq \tau_\epsilon$, we have
\bqs
d_-\leq\chi(\tau)\leq d_-+\epsilon.
\eqs
\begin{lem}
Let $c_\epsilon:=2 \sqrt{(d_-+\epsilon) \alpha}$. For all $\tau \geq \tau_\epsilon$, the following function
\bqs
u_\tau(t,x)=\min\left\{ 1, e^{-\frac{c_\epsilon}{2\chi(\tau)}(x-c_\epsilon t-\tau)}\right\}, \quad t\geq0, \quad x\in\R,
\eqs
is a super-solution for \eqref{eq:main}.
\end{lem}
\begin{proof}
We recall the functional $N$ defined as
\bqs
N(u)=\partial_t u-\chi(x-c_{het}t)\partial_x^2u-\alpha u (1-u).
\eqs
We readily have that $N(1)=0$ and that
\bqs
N\left(e^{-\lambda(x-c_\epsilon t-\tau)}\right)\geq (\lambda c_\epsilon-\chi(x-c_{het}t)\lambda^2-\alpha)e^{-\lambda(x-c_\epsilon t-\tau)},
\eqs
for all $\lambda>0$. Next, we note that for each $x\geq c_\epsilon t+\tau$, we have
\bqs
x-c_{het}t\geq \underbrace{(c_\epsilon-c_{het})}_{>0}t+\tau \geq \tau.
\eqs
As a consequence, for all $x\geq c_\epsilon t+\tau$, using the monotonicity of $\chi$, we get
\bqs
N\left(e^{-\lambda(x-c_\epsilon t-\tau)}\right)\geq (\lambda c_\epsilon-\chi(\tau)\lambda^2-\alpha)e^{-\lambda(x-c_\epsilon t-\tau)}.
\eqs
Now evaluating at $\lambda=\frac{c_\epsilon}{2\chi(\tau)}$, we obtain
\bqs
N\left(e^{-\frac{c_\epsilon}{2\chi(\tau)}(x-c_\epsilon t-\tau)}\right)\geq \frac{c_\epsilon^2-4\chi(\tau)\alpha}{4\chi(\tau)}e^{-\frac{c_\epsilon}{2\chi(\tau)}(x-c_\epsilon t-\tau)}, \quad x\geq c_\epsilon t+\tau.
\eqs
Finally, as $\tau \geq \tau_\epsilon$, we get 
\bqs
c_\epsilon^2-4\chi(\tau)\alpha \geq c_\epsilon^2-4(d_-+\epsilon)\alpha = 0,
\eqs
which concludes the proof.
\end{proof}

As a consequence of the above lemma and the comparison principle, we have that
\bqs
\forall c > c_->c_{het}, \quad \limsup_{t \to +\infty} \ \sup_{ x \geq ct} u(t,x) =0,
\eqs
for $u(t,x)$ solution of \eqref{eq:main} from a compactly supported, nontrivial initial datum $0\leq u_0 \leq 1$. And thus, the rightward spreading speed of \eqref{eq:main} is less than or equal to $ c_-$ when $c_{het}<c_-$ which concludes the proof of Theorem~\ref{theo:chi_dec} in this case.

\subsection{Subcase $c_+<c_{het}$}

When $c_+<c_{het}$, we need to prove that the rightward spreading speed is $c_u^*=c_+$. From Proposition~\ref{prop:gen}, we already have that the spreading speed is less  than or equal to $ c_+$, in the sense that
$$\forall c > c_+, \quad \limsup_{t \to +\infty} \sup_{x \geq ct }  u(t,x)= 0,$$
and thus it remains to provide a sub-solution in this case.

%
%
%


As a matter of fact this sub-solution is valid in all subcases of~\eqref{ass:chi_dec}, and hereafter we simply let $0<c< \min\left(c_+,c_{het}\right)$. Then, one can find $\epsilon>0$ such that $c<2\sqrt{d_+(\alpha-2\epsilon)}<c_+$ together with $\eta_\epsilon>0$ such that
\bqs
(\alpha-\epsilon)u\leq \alpha u(1-u), \quad u\in[0,\eta_\epsilon].
\eqs
We introduce two positive real numbers
\bqs
\beta_+(\epsilon,c):= \frac{\sqrt{4d_+(\alpha-\epsilon)-c^2}}{2d_+}>0, \text{ and } \beta_+(c):=\beta_+(0,c)= \frac{\sqrt{4d_+\alpha-c^2}}{2d_+}>0,
\eqs
together with the following family of functions
\bqs
u_{\tau,\epsilon}^c(t,x):=\left\{
\begin{array}{lc}
\delta_\epsilon \left[ e^{-\frac{c}{2d_+}(x-ct+\tau)}\cos(\beta_+(\epsilon,c)(x-ct+\tau)) +\epsilon \right], & x-ct+\tau \in\Omega_\epsilon(c), \\
0, & \text{ otherwise,}
\end{array}
\right.
\eqs
with $\Omega_\epsilon(c)=\left[-\frac{\pi}{2\beta_+(\epsilon,c)}-z^-_\epsilon(c),\frac{\pi}{2\beta_+(\epsilon,c)}+z^+_\epsilon(c)  \right]$.
Here $\delta_\epsilon>0$ is fixed such that $0\leq u_{\tau,\epsilon}^c(t,x)\leq \eta_\epsilon$ for all $x-ct+\tau \in\Omega_\epsilon(c)$ and can be chosen independent of $c< 2 \sqrt{d_+ (\alpha - 2 \epsilon)}$.  Furthermore, $z^\pm_\epsilon(c)$ are defined through
\bqs
e^{\mp \frac{c}{2d_+}\left( \frac{\pi}{2\beta_+(\epsilon,c)}+z^\pm_\epsilon(c)\right)}\sin(\beta_+(\epsilon,c)z^\pm_\epsilon(c))=\epsilon,
\eqs
with asymptotics
\bqs
z^\pm_\epsilon(c)=\frac{e^{\mp \dfrac{c\pi}{4d_+\beta_+(c)}}}{\beta_+(c)} \epsilon+o(\epsilon), \quad \text{ as } \epsilon \rightarrow0.
\eqs
 
 We first remark that when $x-ct+\tau \in\Omega_\epsilon(c)$, we have
\bqs
x-c_{het}t \in \left[-\frac{\pi}{2\beta_+(\epsilon,c)}-z^-_\epsilon(c)+(c-c_{het})t-\tau,\frac{\pi}{2\beta_+(\epsilon,c)}+z^+_\epsilon(c)+(c-c_{het})t-\tau  \right].
\eqs
Next, as $\chi(-\infty)=d_+$, there exists $A>0$ such that for all $\xi\leq -A$, we get
\bqs
|\chi(\xi)-d_+|\leq \epsilon^2.
\eqs
As a consequence, for all $$\tau>\tau_\epsilon(c):=A+ \frac{\pi}{2\beta_+(\epsilon,c)}+z^+_\epsilon(c),$$ we have $x-c_{het}t\leq-A$ and
\bqs
\left| \left(d_+ -\chi(x-c_{het}t)\right) \partial_x^2 u_{\tau,\epsilon}^c(t,x) \right| \leq \epsilon^2 \delta_\epsilon \left(\frac{c}{2d_+}+\beta_+(\epsilon,c) \right)^2 e^{ \frac{c}{2d_+}\left( \frac{\pi}{2\beta_+(\epsilon,c)}+z^-_\epsilon(c)\right)}:=\epsilon^2 \delta_\epsilon K_\epsilon(c),
\eqs
for all $x-ct+\tau \in\Omega_\epsilon(c)$. 
This implies that for all $\tau>\tau_\epsilon(c)$
\begin{align*}
N(u_{\tau,\epsilon}^c(t,x)) &\leq \left(d_+ -\chi(x-c_{het}t)\right) \partial_x^2 u_{\tau,\epsilon}^c(t,x)-\delta_\epsilon \epsilon(\alpha-\epsilon)\\
& \leq \epsilon \delta_\epsilon \left[(K_\epsilon(c)+1)\epsilon -\alpha\right], \quad x-ct+\tau \in\Omega_\epsilon(c).
\end{align*}
 Notice that 
$$\liminf_{\epsilon \to 0} K_\epsilon (c) >0.$$
As a consequence, we obtain the following lemma.
\begin{lem}\label{lemsubcmincpchet}
Let $0<c< \min\left(c_+,c_{het}\right)$. There is $\epsilon_0(c)>0$ such that for all $\epsilon\in (0,\epsilon_0(c))$, the function $u_{\tau,\epsilon}^c$ is a sub-solution for all $\tau>\tau_\epsilon(c)$.
\end{lem}

Now, using similar arguments as in the proof of Proposition~\ref{prop:gen} we deduce that the solutions spread at least with speed $\min\left(c_+,c_{het}\right)$, which concludes the proof of Theorem~\ref{theo:chi_dec} in the subcase when $c_+ < c_{het}$.
 
\subsection{Subcase $c_-\leq c_{het} \leq c_+$}

When $c_-\leq c_{het} \leq c_+$, our main Theorem~\ref{theo:chi_dec} asserts that the selected rightward spreading speed is precisely the speed of the heterogeneity, that is $c_u^*=c_{het}$. In that case, we cannot rely on Proposition~\ref{prop:gen} to obtain either inequality, and we need to refine our analysis. Nevertheless, we can use Lemma~\ref{lemsubcmincpchet} to obtain in this subcase that the spreading speed is larger than or equal to $c_{het}=\min\left(c_+,c_{het}\right)$, and it only remains to construct a super-solution which spreads at speed $c_{het}$. This is precisely the result of the following lemma.

\begin{lem}\label{lem:subcase_c-_chet_c+}
Let $c_-\leq c_{het} \leq c_+$. Then, there is $\tau_0>0$ such that for all $\tau\geq\tau_0$ the function $\overline{u}_\tau(t,x)=\min\left\{1,e^{-\sqrt{\frac{c_{het}}{2\chi(\tau)}}(x-c_{het}t-\tau)}\right\}$ is a super-solution of~\eqref{eq:main}.
\end{lem}

\begin{proof}
As $c_- \leq c_{het}$ and $\chi(x)\rightarrow d_-$ as $x\rightarrow+\infty$, there exists $\tau_0>0$ such that for all $\tau\geq\tau_0$ one has
\bqs
c_{het}^2-4\alpha \chi(\tau) \geq 0.
\eqs
Next, we have for $\tau\leq x-c_{het}t$ that $\chi (x - c_{het} t) \leq \chi (\tau)$ and 
\bqs
N(\overline{u}_\tau(t,x))\geq \frac{c_{het}^2-4\alpha \chi(\tau)}{4\chi(\tau)}e^{-\sqrt{\frac{c_{het}}{2\chi(\tau)}}(x-c_{het}t-\tau)} \geq 0.
\eqs
This already concludes the proof.
\end{proof}

\section{Construction of super and sub-solutions in case \eqref{ass:chi_inc} and proof of Theorem~\ref{theo:chi_inc}}\label{sec:inc}

In this section, we construct super and sub-solutions in the case where $\chi$ satisfies assumption \eqref{ass:chi_inc}. We first treat the most difficult case when $c_{het} \in \left( c_+,c_{int} \right)$ and then explain how to deal with the remaining two cases.

\subsection{Case $c_{het} \in \left[ c_+,c_{int} \right)$}\label{sec:subcase_chi_inc}

Recall from the discussion in Section~\ref{sec:case_ass_chi_inc} that
$$g(c) = c_{het}^2 - 4 d_+ \left[ \alpha  + \frac{c - \sqrt{c^2 - 4  d_- \alpha}}{2 d_-} (c_{het} - c) \right],$$
and since $c_{het}\in \left[ c_+, c_{int} \right)$, there exists a unique $c \in (c_-, c_+ ]$ such that $g(c)=0$. For convenience, throughout Section~\ref{sec:subcase_chi_inc} we will denote it by $c_u^*$. Our goal is indeed to prove that this $c_u^*$, whose explicit formula is given in Theorem~\ref{theo:chi_inc}, is the spreading speed of the solution in that case.

\subsubsection{Super-solution}

According to Proposition~\ref{prop:gen}, we already know that the spreading is less than or equal to $c_+$. Thus to construct a super-solution here we only need to consider the case when $c_{het} \in (c_+,c_{int})$. Then recall that $c_u^* \in (c_- , c_+) $ is such that $g(c_u^*) =0$, and let any $c \in ( c_u^*, c_{het})$. We introduce the family of (continuous) functions 
\bqs
u_\tau(t,x)=C \times \left\{
\begin{array}{cl}
1, & x\leq c t-\tau,\\
e^{-\lambda(x-c t+\tau)}, & ct-\tau < x < c_{het}t-\tau,\\
e^{-\lambda(c_{het}-c)t}e^{-\mu(x-c_{het}t + \tau)}, & x\geq c_{het}t-\tau,
\end{array}
\right.
\eqs
where $\lambda$, $\mu$ and $\tau$ are positive and will be adjusted as follows, and $C\geq 1$ is also positive but arbitrary. 

\begin{itemize}
\item For $x\in \left( ct-\tau , c_{het}t-\tau \right)$, we compute
\begin{align*}
N(C e^{-\lambda(x-ct+\tau)})&= C \left(\lambda c -\lambda^2 \chi(x-c_{het}t)-\alpha\right)e^{-\lambda(x-c t+\tau)}+ C^2 \alpha e^{-2\lambda(x-c t+\tau)}\\
& \geq C \left(\lambda c_u^* -\lambda^2 d_--\alpha\right)e^{-\lambda(x-c t+\tau)} \\
& \qquad + C\left[\lambda^2(d_- -\chi(-\tau))  +\lambda (c - c_u^*) \right] e^{-\lambda(x-c t +\tau)}, 
\end{align*}
where we used the fact that $\chi$ is nondecreasing and $x-c_{het}t\leq -\tau$. We now select 
$$ \lambda = \lambda (c_u^*) ,$$
where $\lambda (c)$ denotes the smallest positive (thanks to $c > c_u^* > c_-$) solution of
\bqs
\lambda c -\lambda^2 d_--\alpha=0,
\eqs
that is 
\bqs
\lambda(c):=\frac{c-\sqrt{c^2-4d_-\alpha}}{2d_-}.
\eqs
There exists $\tau_0>0$ such that for all $\tau\geq\tau_0$ we have
\bqs
\lambda(c_u^*)^2(d_--\chi(-\tau))+  \lambda(c_u^*)(c - c_u^*)>0,
\eqs
hence $N (u_\tau (t,x) ) >0$ for $x \in (c t - \tau, c_{het} t - \tau)$.
\item For $x > c_{het}t-\tau$, we have that
\begin{align*}
N(C e^{-\lambda(c_{het}-c)t}e^{-\mu(x-c_{het}t+\tau)})&= C\left(-\lambda (c_{het}-c_u^*)+\mu c_{het}-\alpha-d_+\mu^2\right)e^{-\lambda(c_{het}-c)t}e^{-\mu(x-c_{het}t+\tau)}\\
&~~~+ C \left[ \mu^2(d_+-\chi(x-c_{het}t)) + \lambda (c - c_u^*) \right]e^{-\lambda(c_{het}-c)t}e^{-\mu(x-c_{het}t+\tau)}\\
&~~~+C^2 \alpha e^{-2\lambda(c_{het}-c)t}e^{-2\mu(x-c_{het}t+\tau)}.
\end{align*}
The last two terms are nonnegative as $\chi \leq d_+$ and we finally select $\mu>0$ such that
\bqs
-\lambda (c_{het}-c_u^* )+\mu c_{het}-\alpha-d_+\mu^2=0, \text{ with } \lambda=\lambda(c_u^*).
\eqs
That is, we let
\bqs
\mu:=\frac{c_{het}-\sqrt{g(c_u^*)}}{2d_+}=\frac{c_{het}}{2d_+},
\eqs
as $g(c_u^*)=0$, and we get that $N (u_\tau (t,x)) \geq 0$ for $x > c_{het} t - \tau$. We note that the above formula with $\lambda=\lambda(c_u^*)$ and $\mu=\frac{c_{het}}{2d_+}$ writes
\bqs
-\lambda(c_u^* )(c_{het}-c_u^*)=\alpha-\frac{c_{het}^2}{4d_+}=\lambda_\star.
\eqs
\item It remains to prove that there is a negative jump in the derivative at $x=c_{het}t-\tau$, that is
\bqs
0>\partial_x u_\tau(t,(c_{het}t-\tau)^-)>\partial_x u_\tau(t,(c_{het}t-\tau)^+).
\eqs
This is equivalent to show that $\lambda = \lambda(c_u^*)< \mu = \frac{c_{het}}{2d_+}$, where $c_u^* $ is such that $g(c_u^*)=0$. We claim that we have 
\bqq
\lambda(c_u^*)=\frac{c_{het}}{2d_-}-\frac{1}{2d_-}\sqrt{c_{het}^2-4d_-\alpha+4d_-\lambda_\star},
\label{eqlambdac}
\eqq
where $\lambda_\star = \alpha - \frac{c_{het}^2}{4d_+}$. Let us assume that the claim \eqref{eqlambdac} holds. Then, we obtain
\bqs
\lambda = \lambda(c_u^*)=\frac{c_{het}}{2d_-}\left(1-\sqrt{1-\frac{d_-}{d_+}}\right)<\frac{c_{het}}{2d_-}\left(1-\left(1-\frac{d_-}{d_+}\right)\right)=\frac{c_{het}}{2d_+} = \mu,
\eqs
since $0<1-\frac{d_-}{d_+}<1$ and we have verified the fact that there is a negative jump in the derivative at $x=c_{het}t-\tau$. Coming back to the formula \eqref{eqlambdac}, we first note that by definition of $\lambda(c_u^*)$ it solves
\bqs
d_- \lambda(c_u^*)^2-c_u^* \lambda(c_u^*)+\alpha =0,
\eqs
and thus it is also solution of
\bqs
d_- \lambda(c_u^*)^2-c_{het}\lambda(c_u^*)+\alpha =-\lambda(c_u^*)(c_{het}-c_u^*)=\lambda_\star.
\eqs
Finally, the fact that $\lambda(c_u^*)$ is given by \eqref{eqlambdac} (and not the other positive root) can be deduced by comparing the formula in the limiting case $c_{het}=c_+$.
\end{itemize}


\begin{lem}\label{lem:supsolanom} Let $c_+<c_{het}<c_{int}$. Let also $c_u^*$ be such that $g(c_u^*) =0$ and $\lambda(c):=\frac{c-\sqrt{c^2-4d_-\alpha}}{2d_-}$. Then for any $c > c_u^*$, there exists $\tau_0>0$ such that, for each $\tau\geq \tau_0$ and $C\geq 1$,
\bqs
u_\tau(t,x)= C \times \left\{
\begin{array}{cl}
1, & x\leq ct-\tau,\\
e^{-\lambda(c_u^*)(x-ct+\tau)}, & ct-\tau < x < c_{het}t-\tau,\\
e^{-\lambda(c_u^*)(c_{het}-c)t}e^{-\frac{c_{het}}{2d_+}(x-c_{het}t+\tau)}, & x\geq c_{het}t-\tau,
\end{array}
\right.
\eqs
is a super-solution of~\eqref{eq:main}. 
\end{lem}
We refer to Figure~\ref{fig:SupSol} for an illustration.

Choosing $C$ large enough so that $u_\tau (0,\cdot) \geq u_0$, we find that
$$\forall c > c_u^*, \quad \limsup_{t \to +\infty} \ \sup_{x \geq ct}u (t,x) =0.$$
In other words, the solution of \eqref{eq:main} with compactly supported initial datum spreads at speed less than or equal to $c_u^*$.

\begin{figure}[!t]
\centering
\includegraphics[width=0.49\textwidth]{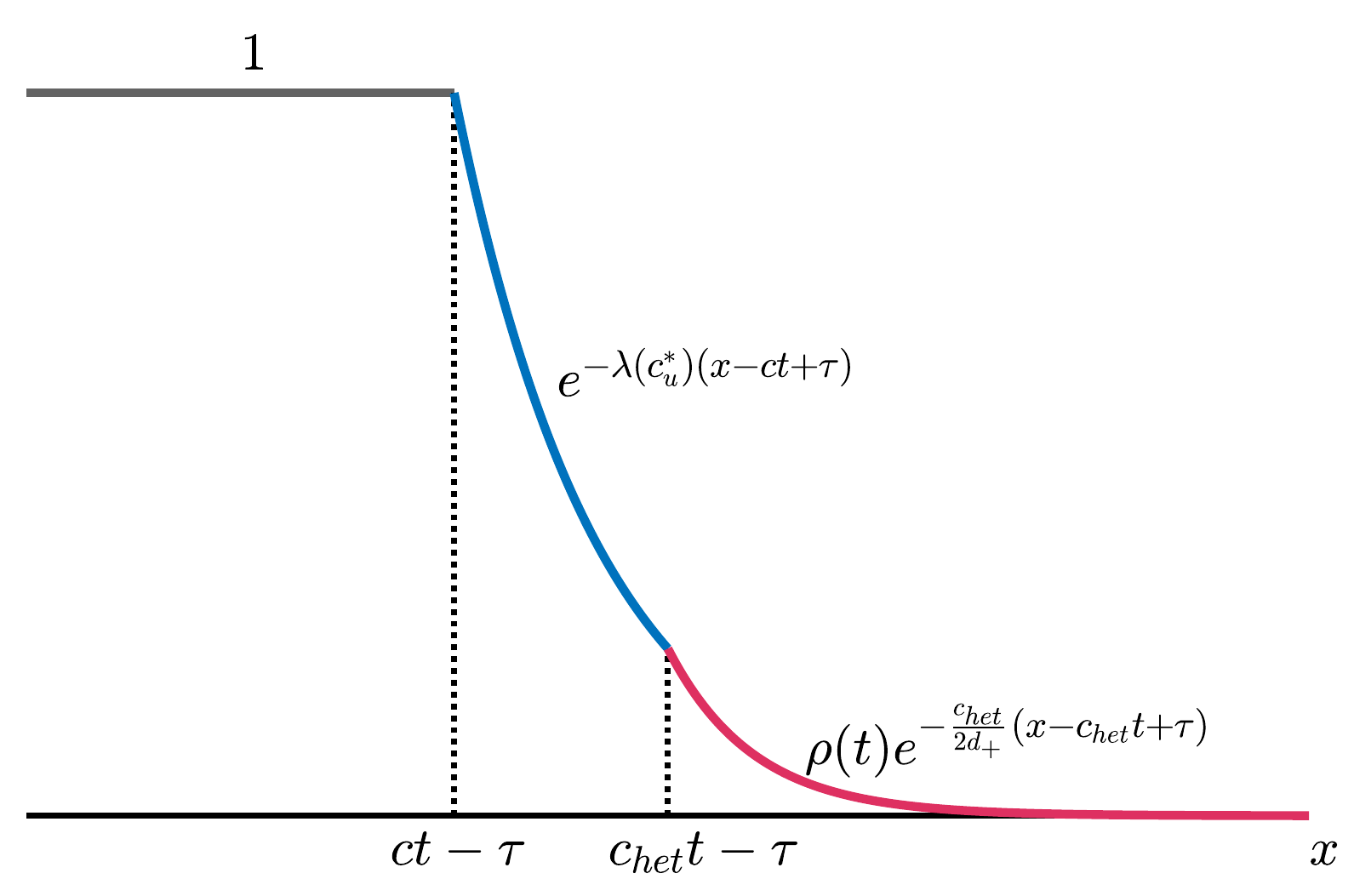}
\caption{Sketch of the super-solution $u_\tau(t,x)$ given in Lemma~\ref{lem:supsolanom} with $C=1$ which is composed of three parts: it is constant and equal to $1$ for $x\leq ct-\tau$ (gray curve), and then it is the concatenation of two exponentials (blue and pink curves) for $x\geq ct-\tau$ which are glued at $x=c_{het}t-\tau$. Note that the factor $\rho(t)$ is 
to ensure continuity between the two exponentials.}
\label{fig:SupSol}
\end{figure}

\subsubsection{Sub-solution}

Here we assume that $c_{het} \in [c_+,c_{int})$. We let $c\in(c_-,c_+)$ be such that $c_-<c<c_u^*=g^{-1}(0)\leq c_+$. We also let $\epsilon >0$ (to be made arbitrarily small) and $\eta>0$ be such that
\bqs
(\alpha-\epsilon)u <\alpha u(1-u), \quad \text{ for } u\in[0,\eta].
\eqs
We are going to construct a suitable sub-solution of
\begin{equation}\label{eq:Msub}
\partial_t u -\chi(x-c_{het}t)\partial_x^2 u -(\alpha-\epsilon)u \leq 0, \quad t \geq 0, \ x \in \R ,
\end{equation}
moving with speed~$c$. Provided that this sub-solution is smaller than $\eta$, then clearly it is also a sub-solution of \eqref{eq:main}. For convenience, we introduce the linear operator
\bqs M(u):=\partial_t u -\chi(x-c_{het}t)\partial_x^2 u -(\alpha-\epsilon)u.
\eqs
We give a first sub-solution of~\eqref{eq:Msub}, which has compact support to the left, and that writes
\bqs
\underline{u}_{1,\tau}(t,x)=\max\left\{0,e^{-\lambda (x-ct+\tau)}- e^{-(\lambda+\gamma)(x-ct+\tau)}\right\},
\eqs
where $\lambda$, $\gamma$ and $\tau$ are positive constants which we adjust below.
\begin{itemize}
\item First, we select $\lambda$ as a root to the equation
\bqs
\lambda^2d_--\lambda c +\alpha-\epsilon =0, \quad  \lambda=\frac{c-\sqrt{c^2-4d_-(\alpha-\epsilon)}}{2d_-}>0.
\eqs 
By our choice of $\lambda$ we will always have 
\bqs M\left(e^{-\lambda(x-c t +\tau)}\right) = \lambda^2 (d_--\chi(x-c_{het}t))e^{-\lambda(x-ct+\tau)} \leq 0 .\eqs


\item Next we pick $\gamma$ such that
$$0< d_- \gamma< c - 2 \lambda d_- = \sqrt{c^2-4d_-(\alpha-\epsilon)}.$$
Then, on the support of $\underline{u}_{1,\tau}$ we have
\bqs
 M(\underline{u}_{1,\tau}(t,x))=\partial_t \underline{u}_{1,\tau}(t,x) -\chi(x-c_{het}t)\partial_x^2 \underline{u}_{1,\tau}(t,x) -(\alpha-\epsilon)\underline{u}_{1,\tau}(t,x),
\eqs
and by linearity we get
\bqs
M(\underline{u}_{1,\tau}(t,x))=M(e^{-\lambda (x-ct+\tau)})- M(e^{-(\lambda+\gamma)(x-ct+\tau)}) \leq - M(e^{-(\lambda+\gamma)(x-ct+\tau)}).
\eqs
Let us also note that 
\bqs
M\left(e^{-(\lambda+\gamma)(x-ct+\tau)}\right)=\left( (\lambda+\gamma)c-(\lambda+\gamma)^2\chi(x-c_{het}t)-(\alpha-\epsilon)\right)e^{-(\lambda+\gamma)(x-ct+\tau)},
\eqs
and using the fact that $\lambda>0$ is such that $\lambda^2d_--\lambda c +\alpha-\epsilon =0$, we can simplify the above expression to
\bqs
M\left(e^{-(\lambda+\gamma)(x-ct+\tau)}\right)=\left( \gamma \left[c-2\lambda d_- - \gamma d_-\right]+(d_--\chi(x-c_{het}t))(\lambda+\gamma)^2\right)e^{-(\lambda+\gamma)(x-ct+\tau)}.
\eqs
From our choice of~$\gamma$, one can find $x_*\in\R$ such that for $x-c_{het}t \leq x_*$ we have
\bqs
\chi(x-c_{het}t)-d_-\leq \frac{\gamma \left[c-2\lambda d_- -\gamma d_- \right]}{(\lambda+\gamma)^2}.
\eqs
As a consequence, for $x-c_{het}t \leq x_*$, we have $M(e^{-(\lambda+\gamma)(x-ct-\tau)})\geq 0$ and so
\bqs
 M(\underline{u}_{1,\tau}(t,x)) \leq 0
\eqs
on $x-c_{het}t \leq x_*$.
\item We now assume that $x-c_{het}t \geq x_*$. From the previous computations we obtain that
\bqs
M (\underline{u}_{1,\tau}(t,x))\leq (d_--\chi(x-c_{het}t))\left( \lambda^2-  (\lambda+\gamma)^2 e^{-\gamma (x-ct+\tau) } \right)e^{-\lambda(x-ct+\tau)}.
\eqs
Recall that $c  < c_+ \leq c_{het}$. It follows that there exists $\tau_*>0$ such that for all $\tau\geq \tau_*$ we have
\bqs
\frac{\lambda^2}{ (\lambda+\gamma)^2} > e^{-\gamma (x-ct+\tau)}, \quad x-c_{het}t\geq x_* ,
\eqs
hence $M (\underline{u}_{1,\tau} (t,x)) < 0$.
\end{itemize}
As a conclusion, we have obtained the following lemma.

\begin{lem}\label{lem:sub_pull1}
Let $c\in(c_-,c_+)$ be such that $c_-<c<c_u^*=g^{-1}(0) \leq c_+$. Then there exists $\tau_* >0$ such that
\bqs
\underline{u}_{1,\tau}(t,x)=\max\left\{0,e^{-\lambda (x-ct+\tau)}- e^{-(\lambda+\gamma)(x-ct+\tau)}\right\},
\eqs
is a sub-solution of~\eqref{eq:Msub} for all $\tau\geq \tau_*$, with 
$$\lambda = \lambda_\epsilon(c):=\frac{c-\sqrt{c^2-4d_-(\alpha-\epsilon)}}{2d_-}>0,$$
and $0< \gamma = \gamma_\epsilon(c)<\frac{1}{d_-} \sqrt{c^2-4d_-(\alpha-\epsilon)}$ for some $\epsilon>0$ small enough.
\end{lem}

We now use the above sub-solution to construct another one which will  be compactly supported. Up to reducing $\epsilon>0$ we assume that $\alpha-\frac{c_{het}^2}{4d_-}<\lambda_\star-\epsilon<\lambda_\star=\alpha-\frac{c_{het}^2}{4d_+}$ and denote $\varphi_{\lambda_\star-\epsilon}$ the solution to 
\bqs
\mathcal{L}\varphi=(\lambda_\star-\epsilon)\varphi, \quad \text{ with } \quad \cL := \chi(x) \partial_x^2  +c_{het} \partial_x  +\alpha,
\eqs 
with prescribed asymptotic expansion
\bqs
\varphi_{\lambda_\star-\epsilon} (x) = e^{\left(-\frac{c_{het}}{2d_-}+\frac{1}{2d_-}\sqrt{c_{het}^2-4d_-\alpha+4d_-(\lambda_\star-\epsilon)}\right)x}\left(1+\mathcal{O}\left(e^{\nu' x}\right)\right), \text{ as } x\rightarrow-\infty,
\eqs 
for some $0 < \nu ' < \nu$, and with damped oscillations at $+\infty$ due to $c_{het}^2 < 4d_+ (\alpha - \lambda_\star +\varepsilon)$. 
We cut-off $\varphi_{\lambda_\star-\epsilon}$ to the right at the smallest point $x_\epsilon\in\R$ where it vanishes and denote
\bqs
\widetilde{\varphi}_{\lambda_\star-\epsilon}(x)=\left\{
\begin{array}{cl}
\varphi_{\lambda_\star-\epsilon}(x), & x\leq x_\epsilon,\\
0, & x> x_\epsilon.
\end{array}
\right.
\eqs

We define, for any $\tau > - 2 x_\epsilon$, the following sub-solution
\bqs
\underline{u}_{2,\tau}(t,x)=\left\{
\begin{array}{cl}
\underline{u}_{1,\tau}(t,x), & x-c_{het}t \leq -\tau/2,\\
c_{\tau}\underline{u}_{1,\tau}(t,c_{het}t-\tau/2)\widetilde{\varphi}_{\lambda_\star-\epsilon}(x-c_{het}t), & x-c_{het}t > -\tau/2,
\end{array}
\right.
\eqs
where $c_\tau = 1/\varphi_{\lambda_\star-\epsilon}(-\tau/2)>0$. With Lemma~\ref{lem:sub_pull1} and the definition of $\widetilde{\varphi}_{\lambda_\star-\epsilon}$, we only need to verify that $\underline{u}_{2,\tau}(t,x)$ is a sub-solution of~\eqref{eq:Msub} for $-\tau/2 \leq x-c_{het}t\leq x_\epsilon$, and also that the jump of the spatial derivative at $x- c_{het} t$ has the correct sign. 

For $- \tau/ 2 < x - c_{het} t < x_\epsilon$, we have
\begin{align*}
M (\underline{u}_{2,\tau}(t,x))& = -\lambda_\epsilon(c)(c_{het}-c)c_\tau e^{-\lambda_\epsilon(c)((c_{het}-c)t+\tau/2)}\varphi_{\lambda_\star-\epsilon}(x-c_{het}t)\\
&~~~+(\lambda_\epsilon(c)+\gamma_\epsilon(c))(c_{het}-c)c_\tau e^{-(\lambda_\epsilon(c)+\gamma_\epsilon(c))((c_{het}-c)t+\tau/2)}\varphi_{\lambda_\star-\epsilon}(x-c_{het}t)\\
&~~~-c_{het}c_{\tau}\underline{u}_{1,\tau}(t,c_{het}t-\tau)\varphi_{\lambda_\star-\epsilon}'(x-c_{het}t)\\
&~~~-\chi(x-c_{het}t)c_{\tau}\underline{u}_{1,\tau}(t,c_{het}t-\tau)\varphi_{\lambda_\star-\epsilon}''(x-c_{het}t)\\
&~~~ -(\alpha-\epsilon) c_{\tau}\underline{u}_{1,\tau}(t,c_{het}t-\tau)\varphi_{\lambda_\star-\epsilon}(x-c_{het}t).
\end{align*}
As $\cL \varphi = (\lambda_\star-\epsilon)\varphi$, we have $\cL \varphi -\epsilon \varphi= (\lambda_\star-2\epsilon)\varphi$ and
\begin{align*}
M (\underline{u}_{2,\tau}(t,x))& =-\lambda_\epsilon(c)(c_{het}-c)c_\tau e^{-\lambda_\epsilon(c)((c_{het}-c)t+\tau/2)}\varphi_{\lambda_\star-\epsilon}(x-c_{het}t)\\
&~~~+(\lambda_\epsilon(c)+\gamma_\epsilon(c))(c_{het}-c)c_\tau e^{-(\lambda_\epsilon(c)+\gamma_\epsilon(c))((c_{het}-c)t+\tau/2)}\varphi_{\lambda_\star-\epsilon}(x-c_{het}t)\\
&~~~-(\lambda_\star-2\epsilon)c_{\tau}\underline{u}_{1,\tau}(t,c_{het}t-\tau)\varphi_{\lambda_\star-\epsilon}(x-c_{het}t):=\mathcal{R}(t,x).
\end{align*}
Next, we recall that
\bqs
\lambda_\star=-\lambda(c_u^*)(c_{het}-c_u^*),
\eqs
such that the right-hand side of the previous inequality can be written as
\bqs
\mathcal{R}(t,x)=\mathcal{D}(t,x)c_\tau e^{-\lambda_\epsilon(c)((c_{het}-c)t+\tau/2)}\varphi_{\lambda_\star-\epsilon}(x-c_{het}t),
\eqs
with
\bqs
\mathcal{D}(t,x)=-\lambda_\epsilon(c)(c_{het}-c)+\lambda(c_u^*)(c_{het}-c_u^*)+2\epsilon + \left[(\lambda_\epsilon(c)+\gamma_\epsilon(c))(c_{het}-c)+\lambda_\star-2\epsilon\right]e^{-\gamma_\epsilon(c)((c_{het}-c)t+\tau/2)}.
\eqs
First, we can pick $\tau\geq \tau_\epsilon$ really large such that
\bqs
 \left[(\lambda_\epsilon(c)+\gamma_\epsilon(c))(c_{het}-c)+\lambda_\star-2\epsilon\right]e^{-\gamma_\epsilon(c)((c_{het}-c)t+\tau/2)}<\epsilon.
\eqs
Then we note that as $c<c_u^* \leq c_{het}$ we have $\lambda(c_u^*)<\lambda(c)$ and \bqs
-\lambda (c) (c_{het}-c)+\lambda(c_u^*)(c_{het}-c_u^*)<0;
\eqs
here we recall that $\lambda (c)$ denotes the smallest of the two solutions of $d_- \lambda^2  - c \lambda + \alpha = 0$. Similarly, $\lambda_\epsilon (c)$ denoted the smallest solution of the same equation with $\alpha$ replaced by $\alpha -\epsilon$. As a consequence, we can chose $\epsilon>0$ small enough such that
\bqs
-\lambda_\epsilon(c)(c_{het}-c)+\lambda(c_u^*)(c_{het}-c_u^*)+3\epsilon<0.
\eqs
It follows that $M (\underline{u}_{2,\tau} (t,x)) \leq 0$ for $-\tau/2 < x - c_{het} t < x_\epsilon$. It remains to deal with the jump of the spatial derivative at $x - c_{het} t$, which we adress after the statement for the resulting sub-solution.
\begin{figure}[!t]
\centering
\includegraphics[width=0.49\textwidth]{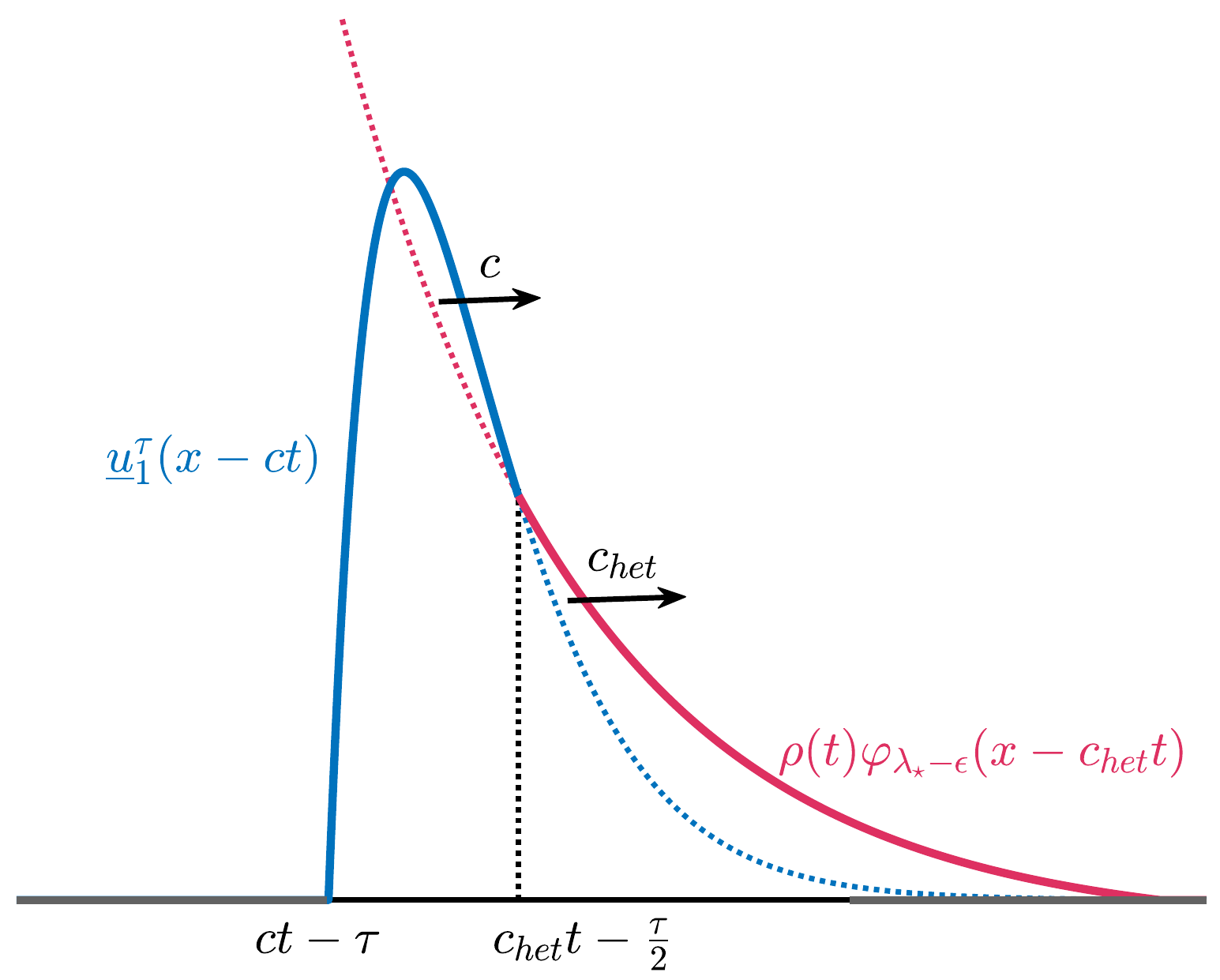}
\caption{Sketch of the sub-solution given in Proposition~\ref{prop:Msub} which is the concatenation of the sub-solution $\underline{u}_1^\tau(x-ct)$ given in Lemma~\ref{lem:sub_pull1} (composed of the difference of two exponentials) and the function ${\varphi}_{\lambda_\star-\epsilon}$ which solves $\cL \varphi = (\lambda_\star-\epsilon)\varphi$ with prescribed asymptotic behavior at $-\infty$. Note that the factor $\rho(t)$ is to ensure continuity at the matching point $x=c_{het}t-\tau/2$.}
\label{fig:SubSolAnom}
\end{figure}

\begin{prop}\label{prop:Msub}
Let $c\in(c_-,c_+)$ be such that $c_-<c<c_u^*=g^{-1}(0)\leq c_+$. Then, there is $\epsilon_0(c )>0$ and $\tau_0(\epsilon,c )>0$ such that for all $0<\epsilon<\epsilon_0(c)$ and all $\tau\geq \tau_0(\epsilon,c)$, we have that
\bqs
\underline{u}_{c,\epsilon,\tau}(t,x)=\left\{
\begin{array}{cl}
\underline{v}_{c,\epsilon,\tau}(t,x), & x-c_{het}t \leq -\tau/2,\\
c_{\epsilon,\tau}\underline{v}_{c,\epsilon,\tau}(t,c_{het}t-\tau/2)\widetilde{\varphi}_{\lambda_\star-\epsilon}(x-c_{het}t), & x-c_{het}t > -\tau/2,
\end{array}
\right.
\eqs
is a sub-solution of \eqref{eq:Msub} for all $(t,x)\in\R^+\times\R$. Here, we have set
\bqs
\underline{v}_{c,\epsilon,\tau}(t,x)=\max\left\{0,e^{-\lambda_\epsilon(c) (x-ct+\tau)}- e^{-(\lambda_\epsilon(c)+\gamma_\epsilon(c))(x-ct+\tau)}\right\},
\eqs
with $\lambda_\epsilon(c)=\frac{c-\sqrt{c^2-4d_-(\alpha-\epsilon)}}{2d_-}>0$ and $0<\gamma_\epsilon(c)< \frac{1}{d_-} \sqrt{c^2-4d_-(\alpha-\epsilon)}$. Moreover, the function $\widetilde{\varphi}_{\lambda_\star-\epsilon}(x-c_{het}t)$ is nonnegative and defined from ${\varphi}_{\lambda_\star-\epsilon}$ which solves $\cL \varphi = (\lambda_\star-\epsilon)\varphi$ with prescribed asymptotic behavior at $-\infty$:
\bqs
\varphi_{\lambda_\star-\epsilon} (x) = e^{\left(-\frac{c_{het}}{2d_-}+\frac{1}{2d_-}\sqrt{c_{het}^2-4d_-\alpha+4d_-(\lambda_\star-\epsilon)}\right)x}\left(1+\mathcal{O}\left(e^{\nu' x}\right)\right), \text{ as } x\rightarrow-\infty,
\eqs
for some $0<\nu'<\nu$. Finally, the normalizing constant $c_{\epsilon,\tau}>0$ is given by $c_{\epsilon,\tau}=1/\varphi_{\lambda_\star-\epsilon}(-\tau/2)$. 
\end{prop}
We refer to Figure~\ref{fig:SubSolAnom} for an illustration.

\begin{proof}
Let $c\in(c_-,c_+)$ be such that $c_-<c<c_u^*=g^{-1}(0) \leq c_+$. Then, there exists $\epsilon_0(c)>0$ such that for all $0<\epsilon<\epsilon_0(c)$, we have
\begin{align*}
&\alpha-\frac{c_{het}^2}{4d_-}<\lambda_\star-\epsilon<\lambda_\star,\\
&\lambda_\epsilon(c)=\frac{c-\sqrt{c^2-4d_-(\alpha-\epsilon)}}{2d_-}>0,\\
&-\lambda_\epsilon(c)(c_{het}-c)+\lambda(c_u^*)(c_{het}-c_u^*)+3\epsilon<0.
\end{align*}
Then, there exists $\tau_0(\epsilon,c)$ such that for all $\tau\geq \tau_0(\epsilon,c)$ one has
\begin{align*}
&\frac{\lambda_{\epsilon}(c)^2}{ (\lambda_{\epsilon}(c)+\gamma_{\epsilon}(c))^2} > e^{-\gamma_{\epsilon}(c) (x-ct+\tau)}, \quad x-c_{het}t\geq x_*(\epsilon,c),\\
&\left[(\lambda_\epsilon(c)+\gamma_\epsilon(c))(c_{het}-c)+\lambda_\star-2\epsilon\right]e^{-\gamma_\epsilon(c)((c_{het}-c)t+\tau/2)}<\epsilon,
\end{align*}
where $x_*(\epsilon,c)\in\R$ is defined such that 
\bqs
\chi(x-c_{het}t)-d_-\leq \frac{\gamma_\epsilon(c) \left[c-2\lambda_\epsilon(c) d_--\gamma_\epsilon(c)\right]}{(\lambda_\epsilon(c)+\gamma_\epsilon(c))^2}, \quad x-c_{het}t \leq x_*(\epsilon,c).
\eqs
According to the above discussions, we already have that $\underline{u}_{c,\epsilon,\tau}$ is a sub-solution on each subdomains $x - c_{het} t < -\tau/2$ and $x-c_{het} t > -\tau/2$. 

It remains to check that there is a positive jump in the spatial derivative of the sub-solution at $x=c_{het}t-\tau/2$, that is
\bqs
\partial_x \underline{u}_{c,\epsilon,\tau}(t,(c_{het}t-\tau/2)^-)<\partial_x \underline{u}_{c,\epsilon,\tau}(t,(c_{het}t-\tau/2)^+)<0.
\eqs
First, we compute:
\bqs
\partial_x \underline{u}_{c,\epsilon,\tau}(t,(c_{het}t-\tau/2)^-)=-\lambda_\epsilon(c) e^{-\lambda_\epsilon(c)\left[(c_{het}-c)t+\tau/2\right]}\left[1-\frac{\lambda_\epsilon(c)+\gamma_\epsilon(c)}{\lambda_\epsilon(c)}e^{-\gamma_\epsilon(c)\left[(c_{het}-c)t+\tau/2\right]}\right],
\eqs
and
\bqs
\partial_x \underline{u}_{c,\epsilon,\tau}(t,(c_{het}t-\tau/2)^+)=\frac{\varphi_{\lambda_\star-\epsilon}'(-\tau/2)}{\varphi_{\lambda_\star-\epsilon}(-\tau/2)}e^{-\lambda_\epsilon(c)\left[(c_{het}-c)t+\tau/2\right]}\left[1- e^{-\gamma_\epsilon(c)\left[(c_{het}-c)t+\tau/2\right]}\right].
\eqs
We now use the prescribed asymptotic behavior at $-\infty$ of $\varphi_{\lambda_\star-\epsilon}$ to get that
\bqs
\frac{\varphi_{\lambda_\star-\epsilon}'(-\tau/2)}{\varphi_{\lambda_\star-\epsilon}(-\tau/2)} \longrightarrow -\frac{c_{het}}{2d_-}+\frac{1}{2d_-}\sqrt{c_{het}^2-4d_-\alpha+4d_-(\lambda_\star-\epsilon)}, \quad \text{ as } \tau \rightarrow +\infty.
\eqs
Recall \eqref{eqlambdac}, that is
\bqs
\lambda(c_u^*)=\frac{c_{het}}{2d_-}-\frac{1}{2d_-}\sqrt{c_{het}^2-4d_-\alpha+4d_-\lambda_\star},
\eqs
and also that $\lambda (c) > \lambda (c_u^*)$ due to $c < c_u^*$. Thus
we can select $\epsilon>0$ even smaller to have
\bqs
-\lambda_\epsilon(c)< -\frac{c_{het}}{2d_-}+\frac{1}{2d_-}\sqrt{c_{het}^2-4d_-\alpha+4d_-(\lambda_\star-\epsilon)}.
\eqs
This implies that upon taking $\tau$ larger we can always ensure that
\bqs
\partial_x \underline{u}_{c,\epsilon,\tau,\eta}(t,(c_{het}t-\tau/2)^-)<\partial_x \underline{u}_{c,\epsilon,\tau,\eta}(t,(c_{het}t-\tau/2)^+)<0.
\eqs
This concludes the proof of Proposition~\ref{prop:Msub}.
\end{proof}

We are now in a position to prove that the solution~$u$ of \eqref{eq:main} with compactly supported initial datum spreads with speed larger than or equal to $c_u^*$. Take any $c \in (c_- ,c_u^*)$ arbitrarily close to $c_u^*$, and notice that $\underline{u}_{c,\epsilon, \tau}$ from Proposition~\ref{prop:Msub} is uniformly bounded from above since it is compactly supported and continuous. In particular, we can find $\delta_0 >0$ small enough so that, for any $0 < \delta \leq \delta_0$,
$$0 \leq \delta \underline{u}_{c,\epsilon,\tau}\leq \eta, $$
where $\eta$ is such that
\bqs
(\alpha-\epsilon)u <\alpha u(1-u), \quad \text{ for } u\in[0,\eta].
\eqs
It is then straightforward to check that, using the linearity of~$M$, 
$$N ( \delta \underline{u}_{c,\epsilon,\tau} ) < \delta  M ( \underline{u}_{c,\epsilon,\tau}) \leq 0,$$
i.e. $\delta \underline{u}_{c,\epsilon,\tau}$ is a sub-solution of~\eqref{eq:main} for any $\delta \leq \delta_0$. Proceeding as in the proof of Proposition~\ref{prop:gen}, we can infer that the solution spreads with speed larger than or equal to $c_u^*$. We omit the details and Theorem~\ref{theo:chi_inc} is proved in the case when $c_{het} \in [ c_+,c_{int})$.

\subsection{Case $c_{het} \geq c_{int}$}

When $c_{het} \geq c_{int}$, we need to prove that the rightward spreading speed is $c_u^*=c_-$. In this case, using the second statement of Proposition~\ref{prop:gen} we have already proved that the solution spreads with speed larger than or equal to $c_-$, and it only remains to construct a super-solution to conclude the proof of Theorem~\ref{theo:chi_dec} in that case. This is precisely the purpose of the next lemma.



\begin{lem} Assume that $c_-<c_+<c_{int} \leq c_{het}$. There exists $\tau_0>0$ such that for each $\tau\geq \tau_0$,
\bqs
u_\tau(t,x)=\left\{
\begin{array}{cl}
1, & x\leq c_-t+\tau,\\
e^{-\frac{c_-}{2d_-}(x-c_-t-\tau)}, & c_-t+\tau < x < c_{het}t+\tau,\\
e^{-\frac{c_-}{2d_-}(c_{het}-c_-)t}e^{-\mu_-(x-c_{het}t-\tau)}, & x \geq c_{het}t+\tau,
\end{array}
\right.
\eqs
is a super-solution of \eqref{eq:main}, with $\mu_-:=\frac{c_{het}+\sqrt{g(c_-)}}{2d_+}>0$.
\end{lem}

\begin{proof}
As $c_{het} \geq c_{int}$, we have that 
$$g(c_-) = c_{het}^2 - 4 d_+ \left( \alpha + \frac{c_-}{2d_-} (c_{het} - c_- ) \right) \geq 0,$$ and $\mu_-$ is well-defined. For $x\in (c_-t+\tau,c_{het}t+\tau)$, we have
\bqs
N\left(e^{-\frac{c_-}{2d_-}(x-c_-t-\tau)}\right)\geq \left[ \frac{c_-^2}{4d_-^2}\left(d_--\chi(\tau)\right)+\alpha e^{\frac{c_-}{2d_-} \tau}\right]e^{-\frac{c_-}{2d_-}(x-c_-t-\tau)}.
\eqs
Thus, we fix $\tau_0>0$ such that for all $\tau\geq \tau_0$
\bqs
\frac{c_-^2}{4d_-^2}\left(d_--\chi(\tau)\right)+\alpha e^{\frac{c_-}{2d_-} \tau}>0.
\eqs
Next, for $x > c_{het}t+\tau$, we have that 
\bqs
N\left( u_\tau(t,x)\right)>\left(-\frac{c_-}{2d_-}(c_{het}-c_-)+\mu_-c_{het}-\alpha -d_+\mu_-^2\right)e^{-\frac{c_-}{2d_-}(c_{het}-c_-)t}e^{-\mu_-(x-c_{het}t-\tau)}=0,
\eqs
as $\mu_->0$ is the  largest  positive root of 
\bqs
-\frac{c_-}{2d_-}(c_{het}-c_-)+\mu c_{het}-\alpha -d_+\mu^2=0.
\eqs
Finally, we have that 
$$\mu_-  > \frac{c_-}{2d_-},$$
since
$$\mu_- \geq \frac{c_{het}}{2d_+} \geq \frac{c_{int}}{2d_+} >  \sqrt{\frac{\alpha}{d_-}} = \frac{c_-}{2d_-}.$$
This insures that the jump of the spatial derivative at $x= c_{het} t + \tau$ has the correct sign and $u_\tau$ is a super-solution of~\eqref{eq:main}. \end{proof}

\subsection{Case $c_{het} < c_+$}

When $c_{het} < c_+$, we need to prove that the rightward spreading speed is $c_u^*=c_+$. In this case, using the first statement of Proposition~\ref{prop:gen} we have already proved that the spreading speed is less than or equal to~$c_+$ and it only remains to construct a sub-solution to conclude the proof of Theorem~\ref{theo:chi_inc}.

We proceed as in the proof of Lemma~\ref{lemsubcmincpchet}. First we let $c_{het}<c<c_+$. Then, one can find $\epsilon>0$ such that $c<2\sqrt{d_+(\alpha-2\epsilon)}<c_+$ together with $\eta_\epsilon>0$ such that
\bqs
(\alpha-\epsilon)u\leq \alpha u(1-u), \quad u\in[0,\eta_\epsilon].
\eqs
We again define
\bqs
\beta_+(\epsilon,c):= \frac{\sqrt{4d_+(\alpha-\epsilon)-c^2}}{2d_+}>0, \text{ and } \beta_+(c):=\beta_+(0,c)= \frac{\sqrt{4d_+\alpha-c^2}}{2d_+}>0,
\eqs
together with the following family of functions
\bqs
u_{\tau,\epsilon}^c(t,x):=\left\{
\begin{array}{lc}
\delta_\epsilon \left[ e^{-\frac{c}{2d_+}(x-ct-\tau)}\cos(\beta_+(\epsilon,c)(x-ct-\tau)) +\epsilon \right], & x-ct-\tau \in\Omega_\epsilon(c), \\
0, & \text{ otherwise,}
\end{array}
\right.
\eqs
with $\Omega_\epsilon(c)=\left[-\frac{\pi}{2\beta_+(\epsilon,c)}-z^-_\epsilon(c),\frac{\pi}{2\beta_+(\epsilon,c)}+z^+_\epsilon(c)  \right]$, $\delta_\epsilon>0$ is fixed such that $0\leq u_{\tau,\epsilon}^c(t,x)\leq \eta_\epsilon$, and $z^\pm_\epsilon(c)$ are defined through
\bqs
e^{\mp \frac{c}{2d_+}\left( \frac{\pi}{2\beta_+(\epsilon,c)}+z^\pm_\epsilon(c)\right)}\sin(\beta_+(\epsilon,c)z^\pm_\epsilon(c))=\epsilon,
\eqs
with asymptotics
\bqs
z^\pm_\epsilon(c)=\frac{e^{\mp \dfrac{c\pi}{4d_+\beta_+(c)}}}{\beta_+(c)} \epsilon+o(\epsilon), \quad \text{ as } \epsilon \rightarrow0.
\eqs
Next, as $\chi(+\infty)=d_+$, there exists $A>0$ such that for all $\xi\geq A$, we get
\bqs
|\chi(\xi)-d_+|\leq \epsilon^2.
\eqs
Proceeding as for Lemma~\ref{lemsubcmincpchet}, we then find that for any $c$, this is a sub-solution provided that $\epsilon$ is small enough. As in the proof of Proposition~\ref{prop:gen}, one can deduce that the solution of \eqref{eq:main} spreads with speed larger than or equal to $c_+$, which ends the proof of Theorem~\ref{theo:chi_inc}. 

\section{Traveling fronts in case \eqref{ass:chi_dec}}\label{sec:TW}

Throughout this section we assume that $c_{het}\in(c_-,c_+)$. Before proceeding with the proof of Theorem~\ref{thmTF}, we introduce the following notion of generalized principal eigenvalue, which can be found in \cite{BR06,BHR07,BR15}, for the elliptic operator $\cL$ defined as
\bqs
\cL = \chi(x) \partial_x^2  +c_{het} \partial_x  +\alpha, \quad x\in\R.
\eqs 
We define $\mu^\star\in\R$ to be
\bqs
\mu^\star:=\sup\left\{ \mu ~|~ \exists \varphi \in \mathscr{C}^2(\R),\, \varphi>0, \, \left(\cL+\mu \right)\leq 0 \right\}.
\eqs
One of the key property of the generalized principal eigenvalue $\mu^\star$ is that it can be obtained as the limit of the Dirichlet principal eigenvalue. More precisely, consider the following Dirichlet problem:
\bqs
\left\{
\begin{split}
\cL \varphi &= -\mu \varphi, \quad |x|<r, \\
\varphi(\pm r)&=0 ,
\end{split}
\right.
\eqs
for each $r>0$ and denote $\mu_d(r)$ the principal eigenvalue given by Krein-Rutman theory \cite{KR}.  Then \cite[Proposition 4.2]{BHR07} ensures that $r\mapsto \mu_d(r)\in \R$ decreases and
\bqs
\mu_d(r)\longrightarrow \mu^\star, \text{ as } r\rightarrow+\infty.
\eqs
We claim that we have the following result.
\begin{lem}\label{lem:signgp}
When $c_{het}\in(c_-,c_+)$ and $\chi$ satisfies \eqref{ass:chi_dec}, then the principal eigenvalue of $\cL$ satisfies $\mu^\star<0$.
\end{lem}
\begin{proof}
We check that the conditions of \cite[Theorem 4.3]{BHR07} are satisfied in our case which translate into our setting by checking that the function $q(x):=4\chi(x)\alpha -c_{het}^2$ is above a fixed positive constant on some (large) interval. As $c_{het}\in(c_-,c_+)$, we have that $q(-\infty)=4d_+\alpha -c_{het}^2>0$. As a consequence, there exists $\epsilon>0$ and $x_0<0$ such that for all $x\leq x_0$, we have
\bqs
q(x)=4\chi(x)\alpha -c_{het}^2\geq\epsilon,
\eqs
which implies that $\mu^\star<0$ from \cite[Theorem 4.3]{BHR07}.
\end{proof}

\paragraph{Existence.} In order to prove the existence of traveling front solutions, we are first going to construct generalized sub and super-solutions for \eqref{TWeq}. The construction of the sub-solution relies on the aforementioned properties of the generalized principal eigenvalue $\mu^\star$. For each $r>0$, we denote by $\varphi_r$ the corresponding eigenfunction to the Dirichlet principal eigenvalue $\mu_d(r)$ which satisfies $\varphi_r>0$ and normalized with $\varphi_r(0)=1$. As $\mu^\star<0$ from Lemma~\ref{lem:signgp}, there exists some $R>0$ such that for any $r\geq R$ we also have $\mu_d(r)<0$. As a consequence, there exists $0<\kappa_r<1$ small enough such that
\bqs
(\alpha+\mu_d(r))(\kappa_r \varphi_r) \leq \alpha \kappa_r \varphi_r (1- \kappa_r \varphi_r).
\eqs

Thus, if one defines $\underline{U}_r$ the following family of functions
\bqs
\underline{U}_r(x) =
\left\{
\begin{array}{cl}
\kappa_r \varphi_r(x), & |x| \leq r, \\
0 , & \text{ otherwise,}
\end{array}
\right.
\eqs
then for all $r\geq R$, the function $\underline{U}_r$ is a generalized sub-solution to \eqref{TWeq}.

Next, as $c_{het}\in(c_-,c_+)$, there exists $\epsilon_0>0$ small enough such that for each $0<\epsilon<\epsilon_0$ one has $c_-<2\sqrt{(d_-+\epsilon)\alpha}<c_{het}$. For such an $\epsilon$, one can find $\tau_\epsilon>0$ such that for any $\tau\geq \tau_\epsilon$ we have
\bqs
d_-\leq \chi(x)\leq d_-+\epsilon, \quad x\geq \tau.
\eqs
We now introduce $\overline{U}_{\epsilon,\tau}$ defined as
\bqs
\overline{U}_{\epsilon,\tau}(x)=\max\left\{1,e^{-\lambda_\epsilon(x-\tau)}\right\}, \quad \lambda_\epsilon=\frac{c_{het}+\sqrt{c_{het}^2-4(d_-+\epsilon)\alpha}}{2(d_-+\epsilon)}>0.
\eqs
For all $\tau\geq \tau_\epsilon$, one can check that $\overline{U}_{\epsilon,\tau}$ is a generalized super-solution to \eqref{TWeq}. Up to further reducing~$\kappa_r$ (or taking~$\tau$ larger), we can always ensure that
\bqs
0\leq \underline{U}_r  \leq  \overline{U}_{\epsilon,\tau} \leq 1, 
\eqs
for some $r\geq R$, $\epsilon\in(0,\epsilon_0)$ and $\tau\geq \tau_\epsilon$.

Now denote by $u$ the solution of 
\begin{equation}\label{TWeq_parab}
\partial_t u = \chi (x) \partial_x^2 u  + c_{het} \partial_x u + \alpha u (1-u), 
\end{equation}
with the initial condition
$$u (t=0, \cdot) \equiv \overline{U}_{\epsilon,\tau}.$$
Since $\overline{U}_{\epsilon, \tau}$ is a super-solution of \eqref{TWeq}, hence also of~\eqref{TWeq_parab}, it follows from parabolic comparison principles that $u$ is nonincreasing in the time variable. Therefore it converges to some function $U$ as $t \to +\infty$, and by parabolic estimates we get that~$U$ is a solution of~\eqref{TWeq}.

Moreover, by construction and another use of the comparison principle, we get that 
\bqs
0\leq \underline{U}_r  \leq U \leq   \overline{U}_{\epsilon,\tau} \leq 1, \text{ on } \R.
\eqs
Using the strong maximum principle, we actually get that $0<U<1$. Indeed, assume there is $x_0\in\R$ for which $U(x_0)=0$, then the strong maximum principle implies that $0\leq \underline{U}_r  \leq U  \equiv 0$, which is impossible. A similar argument holds for the other inequality. 

We also remark that $U$ must be nonincreasing. Indeed, notice that 
$$\partial_x u (t=0) = \overline{U}_{\epsilon,\tau} '  \leq 0.$$
Moreover, derivating~\eqref{TWeq_parab}, we get that $v =\partial_x u$ solves
$$\partial_t v =  \chi (x) \partial_x^2 v + ( c_{het} + \chi ' (x) ) \partial_x v +  \alpha (1-2 u) v.$$
By another comparison principle, we find that 
$$ v (t,x) \leq 0,$$
for all $t >0$ and $x \in \R$, hence $U ' \leq 0$.
 
This latter fact combined with $0\leq \underline{U}_r  \leq U$ ensures that $ U (-\infty )>0$, and one necessarily gets that $U (-\infty)=1$ as $U$ is solution of the ODE \eqref{TWeq}. Next, using the fact that $\overline{U}_{\epsilon,\tau}(x)\rightarrow0$ as $x\rightarrow+\infty$, we get that $U (+\infty)=0$ by comparison. We claim that we actually have $U '<0$ on $\R$. This property is satisfied near $-\infty$ as there exists a unique stable direction. We let $x_*\in\R$ be such that $U '(x_*)=0$ and $U'(x)<0$ for all $x<x_*$. Then, we have
\bqs
\chi(x_*)U ''(x_*)=-\alpha U(x_*)(1- U(x_*))<0,
\eqs
which is a contradiction. 

\paragraph{Asymptotic behavior at $+\infty$.}  As $c_{het}>c_-$ we have that $\frac{\sqrt{c_{het}^2-4d_-\alpha}}{2d_-}>0$, and upon eventually reducing $\epsilon_0>0$, we can ensure that for all $\epsilon\in(0,\epsilon_0)$,
\bqs
\lambda_{\epsilon}>\frac{c_{het}}{2d_-}.
\eqs
As a consequence, we have that for all $x\geq \tau $
\bqs
0<e^{\frac{c_{het}}{2d_-}x} U(x)\leq  e^{\frac{c_{het}}{2d_-}x} \overline{U}_{\epsilon,\tau}(x)=e^{-\left(\lambda_\epsilon-\frac{c_{het}}{2d_-}\right) x+\lambda_\epsilon \tau}.
\eqs
We now prove that~$U$ has the strong exponential decay given in Theorem~\ref{thmTF}. Near $+\infty$, system~\eqref{TWeq} can be written in condensed form
\bqq\label{systemTW}
{\bf U} '(x) = A(x){\bf U}(x)+N(x,{\bf U}(x)), \quad x\in\R,
\eqq
with ${\bf U}(x)=(U(x),V(x))^\mathbf{t}$ and
\bqs
A(x):=\left(\begin{matrix}
0 & 1 \\
-\frac{\alpha}{\chi(x)} & -\frac{c_{het}}{\chi(x)}
\end{matrix}\right), \quad N(x,{\bf U}):=\left(\begin{matrix}
0  \\
\frac{\alpha}{\chi(x)} U^2
\end{matrix}\right).
\eqs

As $\chi$ converges at an exponential rate at $+ \infty$, so does $A(x)$, and since there is a gap between the strong stable and weak stable eigenvalues of $A_\infty:=\underset{x\rightarrow+\infty}{\lim}~A(x)$, the constant coefficient asymptotic system has an exponential dichotomy. By classical arguments on the roughness of exponential dichotomies \cite{coppel}, the non-autonomous system inherits one with the same decay rates as $\chi$ converges at an exponential rate at $+ \infty$. Note that the strong stable eigenvalue is precisely given by $-\lambda_s$ defined in Theorem~\ref{thmTF} while the weak stable eigenvalue is $-\lambda_w$ given by
\bqs
0<\lambda_w:=\frac{c_{het}-\sqrt{c_{het}^2-c_-^2}}{2d_-}<\lambda_s.
\eqs
As a consequence, since we have $0<e^{\frac{c_{het}}{2d_-}x} U(x)\leq  e^{-\left(\lambda_\epsilon-\frac{c_{het}}{2d_-}\right) x+\lambda_\epsilon \tau}$ for all $x\geq \tau$, we deduce that necessarily 
\bqs
U (x) \underset{x \rightarrow +\infty}{\sim} \gamma_s e^{-\lambda_s x},
\eqs
for some $\gamma_s>0$.

\paragraph{Uniqueness.} We first prove that when $c_{het}> c_- = 2\sqrt{d_-\alpha}$ solutions of \eqref{TWeq}-\eqref{TWlim} are unique in $H^1_{\frac{c_{het}}{2d_-}}(\R):=\left\{ U 
~|~ e^{\frac{c_{het}}{2d_-} \cdot} U \in H^1(\R)\right\}$. Let $U \in H^1_{\frac{c_{het}}{2d_-}}(\R)$ and $V \in H^1_{\frac{c_{het}}{2d_-}}(\R)$ be two solutions of \eqref{TWeq}-\eqref{TWlim} and assume by contradiction that $U \not\equiv V$. Without loss of generality, we may assume that $U(x)>V(x)$ on some interval $(a,b)\subset\R$ with $U(a)=V(a)$ and $U(b)=V(b)$. Note that $a,b\in \R \cup \left\{\pm\infty\right\}$. Multiplying the equation for $U$ with $e^{\beta(x)}\frac{V(x)}{\chi(x)}$ and the equation for $V$ with $e^{\beta(x)}\frac{U(x)}{\chi(x)}$, where we set $\beta(x):=\int_0^x \frac{c_{het}}{\chi(y)}\mathrm{d} y$, we obtain
\begin{equation*}
\left\{
\begin{split}
0&=V(x) \frac{\mathrm{d}}{\mathrm{d} x} \left( e^{\beta(x)}U'(x) \right)+\frac{\alpha}{\chi(x)}U(x)V(x)(1-U(x)) e^{\beta(x)},\\
0&=U(x) \frac{\mathrm{d}}{\mathrm{d} x} \left( e^{\beta(x)}V'(x)\right)+\frac{\alpha}{\chi(x)}U(x)V(x)(1-V(x)) e^{\beta(x)}.
\end{split}
\right.
\end{equation*}
As a consequence, we have
\begin{equation*}
\int_a^b V(x) \frac{\mathrm{d}}{\mathrm{d} x} \left( e^{\beta(x)}U'(x) \right)-U(x) \frac{\mathrm{d}}{\mathrm{d} x} \left( e^{\beta(x)}V'(x)\right)\mathrm{d} x=\alpha \int_a^b \frac{U(x)V(x)}{\chi(x)}(U(x)-V(x))e^{\beta(x)}\mathrm{d} x.
\end{equation*}
Note that the above integrals are convergent as on the one hand $U,V \in H^1_{\frac{c_{het}}{2d_-}}(\R)$ which ensures integrability when $b=+\infty$, and on the other hand $\beta(x)\sim \frac{c_{het}}{d_+}x$ as $x\rightarrow -\infty$ which ensures integrability when $a=-\infty$ (recall that $U(-\infty) = V (-\infty) =1$, and the convergence must be exponential by classical arguments on \eqref{systemTW}). Integrating by parts the integral on the left-hand side of the equality, we obtain
\begin{equation*}
\left.e^{\beta(x)}\left(V(x)U'(x)-U(x)V'(x)\right)\right|_a^b=\alpha \int_a^b \frac{U(x)V(x)}{\chi(x)}(U(x)-V(x))e^{\beta(x)}\mathrm{d} x.
\end{equation*}
When both $a,b\in\R$, we have that the right-hand side is strictly positive while the left-hand side is negative using that $U(x)=V(x)$ at $x\in\left\{a,b\right\}$ and $U(x)>V(x)$ for $x\in(a,b)$. If $a=-\infty$, we have that
\begin{equation*}
e^{\beta(x)}\left(V(x)U'(x)-U(x)V'(x)\right) \rightarrow 0, \quad x\rightarrow -\infty,
\end{equation*}
and the left-hand side is negative. If $b=+\infty$, we have that
\begin{align*}
e^{\beta(x)}\left|V(x)U'(x)\right|=e^{\int_0^{x} c_{het} \left(\frac{1}{\chi(y)}-\frac{1}{d_-}\right) \mathrm{d}y}e^{\frac{c_{het}}{d_-}x}\left|V(x)U'(x)\right| \rightarrow 0, \quad x\rightarrow +\infty,\\
e^{\beta(x)}\left|V'(x)U(x)\right|=e^{\int_0^{x} c_{het} \left(\frac{1}{\chi(y)}-\frac{1}{d_-}\right) \mathrm{d}y}e^{\frac{c_{het}}{d_-}x}\left|V'(x)U(x)\right| \rightarrow 0, \quad x\rightarrow +\infty,
\end{align*}
and the left-hand side is again negative. Thus, we have reached a contradiction, and the solution of \eqref{TWeq}-\eqref{TWlim} with strong exponential decay at $+\infty$, if it exists, is unique.

\paragraph{Non existence of solutions of \eqref{TWeq}-\eqref{TWlim} when $c_{het}< c_- $.}

Using the fact that $\chi(+\infty)=d_->0$, we readily obtain the necessary condition that $c_{het}\geq   2\sqrt{d_-\alpha} = c_-$ for the corresponding solution to remain positive near the equilibrium $u=0$. 

\paragraph{Non existence of solutions of \eqref{TWeq}-\eqref{TWlim} when $c_{het}> c_+ $ with strong exponential decay.} Let us assume that $c_{het}> c_+ = 2\sqrt{d_+\alpha}$. We are going to prove that any solution satisfies 
\begin{equation}
U'(x)>-\frac{c_{het}}{2d_+}U(x), \quad x\in\R.
\label{eqweak}
\end{equation}
We know that the above inequality holds true near $-\infty$, and suppose by contradiction that $x_*\in\R$ is the first time where
\begin{equation*}
\left\{
\begin{split}
0&=\chi(x_*) U''(x_*)+c_{het}U'(x_*)+\alpha U(x_*)(1-U(x_*)),\\
U'(x_*)&=-\frac{c_{het}}{2d_+}U(x_*).
\end{split}
\right.
\end{equation*}
Using the fact that $\alpha U (1-U) \leq \alpha U < \frac{c_{het}^2}{4d_+} U$, we have that
\begin{align*}
\chi(x_*) U''(x_*)&=-c_{het} U'(x_*)-\alpha U(x_*)(1-U(x_*))\\
&>-c_{het}U'(x_*)-\frac{c_{het}^2}{4d_+}U(x_*),\\
& =-\frac{c_{het}}{2}U'(x_*),
\end{align*}
from which we deduce that $U''(x_*)>-\frac{c_{het}}{2d_+}U'(x_*)$ as $\chi(x_*)\leq d_+$ and $U'(x_*)<0$. This is a contradiction as $x_*$ is the first time where $U'(x_*)=-\frac{c_{het}}{2d_+}U(x_*)$. Therefore \eqref{eqweak} must hold for all~$x$.  Integrating \eqref{eqweak} we obtain that
\begin{equation*}
U(x)>e^{-\frac{c_{het}}{2d_+}x}U(0), \quad x>0,
\end{equation*}
and thus the solution $U$, if it exists, has weak exponential decay near $+\infty$. As a consequence, as $0<\frac{c_{het}}{2d_+}<\frac{c_{het}}{2d_-}<\lambda_s$, when $c_{het}>2\sqrt{d_+\alpha}$ there cannot exist a solution of \eqref{TWeq}-\eqref{TWlim} with strong exponential decay given by $\lambda_s$.

\section*{Acknowledgements} GF acknowledges support from an ANITI (Artificial and Natural Intelligence Toulouse Institute) 
Research Chair and from Labex CIMI under grant agreement ANR-11-LABX-0040.  The research of MH was partially supported by the National Science
Foundation (DMS-2007759).

%
%
%
%

\bibliographystyle{abbrv}
\bibliography{hetdiff}

\end{document}